\documentclass{amsart}

\usepackage{myPreamble}
\author{Madie Farris}
\address{Department of Mathematics and Computer Science\\Moravian University\\1200 Main Street\\Bethlehem, PA 18018\\ USA}
\email{farrism02@moravian.edu}

\author{Philipp Hieronymi}
\address{Mathematical Institute\\ University of Bonn\\
Endenicher Allee 60\\ 53115 Bonn\\ Germany}
\email{hieronymi@math.uni-bonn.de}

\title{Uniform bounds in d-minimal structures}
\date{\today}

\begin{document}

\begin{abstract}
Let $\R$ be an expansion of the real field such that every subset of $\RR$ definable in $\R$ either has interior or is a finite union of discrete sets. Answering a question by Chris Miller, we show that 
for every $n\in \NN$ and every definable subset $A\subseteq \RR^{n+1}$ there is $N\in \NN$ such that for all $x\in \RR^n$ either $A_x$ has interior or is the union of $N$ discrete sets.
\end{abstract}
\maketitle

\section{Introduction}

Let $\R$ be an expansion of the real field $(\RR,<,+,\cdot)$. When we say a set $X\subseteq \RR^n$ is definable, we mean $X$ is definable in $\R$, possibly with parameters. We say $\R$ is \textbf{o-minimal} if every definable subset of $\RR$ is a finite union of intervals and points.  O-minimality was isolated by van den Dries \cite{vdD-tarski} to generalize key results from semi-algebraic geometry, and developed by Pillay and Steinhorn \cite{PSI} as a tameness notion in the setting of dense linear orders.  The last decade has seen a surge in applications, ranging from number theory and algebraic geometry to theoretical physics and machine learning. \newline

One of the foundational results in the study of o-minimality is that the finiteness condition in the definition of o-minimality holds uniformly in the parameters used to define the set. Indeed, the expansion $\R$ is said to be \textbf{strongly o-minimal} if for every $n\in \NN$ and every definable subset $A\subseteq \RR^{n+1}$ there is $N\in \NN$ such that for all $x\in \RR^n$ the fiber $A_x$ is a union of $N$ intervals and points. Using cell decomposition, Knight, Pillay, and Steinhorn \cite{PSII,PSIII} show that o-minimality implies strong o-minimality in arbitrary o-minimal structures.\\

In this paper, we extend this uniformity to a large class of structures introduced by Miller in \cite{Miller-tame}. 
We say that $\R$ is \textbf{d-minimal} if every definable subset of $\RR$ is a union of an open set and finitely many discrete sets, and that $\R$ is \textbf{strongly d-minimal} if for every $n\in \NN$ and every definable subset $A\subseteq \RR^{n+1}$ there is $N\in \NN$ such that for all $x\in \RR^n$ either $A_x$ has interior or is the union of $N$ discrete sets. The first known examples of strongly d-minimal (yet not o-minimal) expansions of $\RR$ are expansions by a cyclic multiplicative subgroup of $\RR_{>0}$ in \cite{vdd-Powers2}, although they predate the definition of d-minimality. Miller and Speissegger observe (see \cite[Corollary after Theorem 3.4.2]{Miller-tame}) that for $\omega \in \RR_{>0}$ the real field expanded by the logarithmic spiral
\[
S_{\omega} := \{  (e^t\cos(\omega t),e^t \sin(\omega t) \ : \ t \in \RR \}
\]
is strongly d-minimal. As explained in \cite{Miller-linear}, such structures play a crucial role in the study of expansions of the real field as the only examples of tame, yet not o-minimal, expansions by a locally closed trajectory of a vector field. For more examples of strongly d-minimal expansions see Friedman and Miller \cite{FM-Sparse,Miller-fast} and Miller and Tyne \cite{Miller-iteration}.\\

As noted, all the above examples are strongly d-minimal. Indeed, it has been a long standing open question, first raised in \cite{Miller-tame}, whether there is as an expansion of $\RR$ that is not strongly d-minimal, but d-minimal. We answer this question here.

\begin{thm}\label{thm:uniform2}
If $\R$ is d-minimal, then $\R$ is strongly d-minimal.
\end{thm}

What we call \emph{strong d-minimality} here is called \emph{d-minimality} in \cite{Miller-tame}. The reason for dropping \emph{strong} in the original definition might have been that the uniformity was necessary to prove the first cell decomposition results for such structures, and it was far from clear whether a result like Theorem \ref{thm:uniform2} would actually be true.\newline

While (strong) d-minimality has not been studied as extensively as o-minimality, a countable cell decomposition theorem is known (see \cite[Theorem 4]{Miller-tame} and Thamrongthanyalak \cite[Theorem B]{Athipat-Michael}): if $\R$ is strongly d-minimal, then every definable subset of $\RR^n$ is a union of countably many cells. In contrast to the o-minimal case, the uniform bounds are used as an assumption and not obtained as a conclusion. Here, we prove a countable cell decomposition theorem for d-minimal structures that allows us to deduce the uniformity as part of its proof.\newline

Cells are defined as in o-minimal structures (see \cref{defn:cells}). We then introduce \textbf{stacks} as finite disjoint unions of cells of the same type (see \cref{defn:stacks}) such that these cells are also the connected components of the stack. Thus stacks have finitely many connected components, and their definitions make perfect sense in o-minimal structures. However, in d-minimal structures we have to account for sets with countably many connected components. To do so, we introduce grids. Let $D\subseteq (0,1]$ be a discrete set such that $1\in D$ and $D\cup \{0\}$ is closed. We call such a set a \textbf{sequence set}. Then a \textbf{grid} is a definable family $(S_d)_{d\in D}$ of stacks satisfying certain further conditions (see \cref{defn:grid}). In particular, for every $d,e\in D$ with $d>e$, we require that $S_d \subseteq S_e$ and every connected component of $S_d$ is also a connected component of $S_e$. The \textbf{surface} of a grid is defined as the union $\bigcup_{d\in D} S_d$.  A collection $\mathcal{A}$ of subsets of $\RR^n$ is \textbf{compatible} with a collection $\mathcal{B}$ of subsets of $\RR^n$ if for
every $A \in \mathcal{A}$ and every $B\in \mathcal{B}$, either $A\subseteq B$ or $A\cap B = \emptyset$.

\begin{thm}[Grid decomposition theorem]\label{thm:intro2}
Suppose $\R$ is d-minimal. Let $\mathcal{A}$ be a finite collection of definable subsets of $\RR^n$. Then there is a finite partition of $\RR^n$ into surfaces of grids compatible with $\mathcal{A}$.
\end{thm}

Since the surface of a grid is a countable union of cells, it follows that every definable set in a d-minimal structure is a countable union of cells. We also reprove that every d-minimal structures admits a countable (cell) decomposition as defined in \cite[Section 4]{Miller-tame}. It is worth noting that surfaces of grids are not special manifolds in the sense of \cite{Miller-tame}. We discuss this distinction and the connection to earlier decomposition results in detail in Section \ref{section:special}.
\newline

At the beginning of Section \ref{section:Main-Theorem} we collect various consequences of the grid decomposition theorem. Here we announce one that we expect to be of particular importance. 

\begin{thm}[Approximation theorem]\label{thm:approx}
Suppose $\R$ is d-minimal. Let $A\subseteq \RR^n$ be definable. Then there is a sequence set $D$ and a decreasing definable family $(A_d)_{d \in D}$ such that $A = \bigcup_{d \in D} A_d$ and for every $d \in D$
 \begin{enumerate}
     \item $A_d$ has finitely many connected components, and
     \item $\dim A = \dim A_d$.
 \end{enumerate}
\end{thm}
Here $\dim(A)$ is the maximum $m \in \mathbb{N}\cup\{0\}$ for which there is a coordinate projection $\pi:\RR^n\to \RR^m$ such that $\pi(A)$ has interior, if $A$ is nonempty, and $-1$ otherwise.
Similar results, for example Tychonievich \cite[Corollary 4.1.7]{tycho-thesis}, have only been shown for particular examples of d-minimal structures such as the expansion of the real field by $S_{\omega}$. The $A_d$'s in the statement of \cref{thm:approx} satisfy some properties of definable sets in o-minimal structures, and hence sets definable in d-minimal structure can be uniformly approximated by very well-behaved sets. An open question is whether each $A_d$ can be chosen to actually be definable in an o-minimal structure.\newline

All results are stated and proven only for expansions of the real field. However, we expect that our results hold in the larger generality of definably complete expansions of ordered fields, but non-trivial adjustments are necessary. In this generality, finiteness has to be replaced by pseudo-finiteness, and the definition of stacks has to allow for infinitely many connected components. For that reason, the argument in the case over $\RR$ is more straightforward and we decided to concentrate on it here. We leave the more involved argument in the case of definably complete expansions to a future paper.

\subsection*{Acknowledgments} The authors were partially supported by NSF grant DMS-1654725. The second author was partially supported  by the Deutsche Forschungsgemeinschaft (DFG, German Research Foundation) under Germany's Excellence Strategy – EXC-2047. 

\section{Preliminaries}\label{section:preliminaries}

Let $i,j,k,m,$ and $n$ denote natural numbers.

\subsection{Topology}
We first fix notations and recall basic results from point-set topology.\\

We consider $\RR^n$ as a topological space with the usual Euclidean topology.
For convenience, we often add two endpoints $-\infty$ and $+\infty$ to $\RR$ satisfying $-\infty < a < \infty$ for all $a\in \RR$. For $x\in \RR^n$ and $\ep>0$, let $B_\ep(x)$ be the open box given by
\[
B_{\ep}(x):=(x_1-\ep,x_1+\ep)\times \cdots \times (x_{n}-\ep,x_{n}+\ep).
\]
For sets $A\subseteq \RR^m$ and $B\subseteq \RR^n$ and a function $f:A \to B$ we denote the graph of $f$ by $\Gamma(f)$.\\

Let $A,B\subseteq \RR^n$ be such that $A\subseteq B$. We consider $B$ as a topological space with the topology on $B$ induced by the Euclidean topology of $\RR^n$. We use $\cl_B(A)$ to denote the closure of $A$ in $B$, we use $\operatorname{int}_B(A)$ for the interior of $A$ in $B$, we use $\operatorname{bd}_B(A)$ for the boundary $\cl_B(A) \setminus \operatorname{int}_B(A)$ of $A$ in $B$, and we use $\fr_B(A)$ for the frontier $\cl_B(A) \setminus A$ of $A$ in $B$. When $B=\RR^n$, we suppress the index and simply write $\cl(A),\operatorname{int}(A),\operatorname{bd}(A),$ and $\fr(A)$. We say $A$ is \textbf{locally closed} if for every $x\in A$ there is an open neighborhood $U$ of $x$ such that $A \cap U$ is closed in $U$.\newline

For $\varepsilon\in \RR_{>0}$, we define $\isol(A,\varepsilon)$ as
\[
\{ a \in A \ : \ B_{\varepsilon}(a)\cap A = \{a\}\}.
\]
We denote by $\isol(A)$ the set of isolated points of $A$, that is,
\[
\isol(A):= \bigcup_{\varepsilon>0} \isol(A,\varepsilon).
\]
Let $A\subseteq \RR^m \times \RR^n$, and let $x \in \RR^m$. The \textbf{fiber} $A_x$ of $A$ over $x$ is defined as 
\[
A_x := \{ y \in \RR^n \ : \ (x,y) \in A\}.
\]

Let $A\subseteq \RR$ be open. We denote by $\Midp(A)$, the set of midpoints of bounded connected components of $A$. If $A$ is definable in $\R$, so is $\Midp(A)$.\\

Let $w = (w_1, \dots, w_n) \in \{0, 1\}^n$ be a tuple. Set $k:=w_1 + \dots + w_n$ and let
$\lambda(1) < \cdots < \lambda(k)$ be exactly the indices in $\{1,\dots,n\}$ such that $w_{\lambda(j)} = 1$. Let $\pi_w \colon \RR^n \to \RR^k$
be the coordinate projection mapping
\begin{equation*}
(x_1, \dots, x_n) \mapsto (x_{\lambda(1)}, \dots, x_{\lambda(k)}).
\end{equation*}
When we write $\pi$ without an index, we always mean $\pi_{(1,\dots,1,0)}$; that is
the projection onto the first $n-1$ coordinates.\newline 

Let $A\subseteq \RR^n \times \RR$. We say $A$ is \textbf{vertically bounded} if there is a bounded open interval $I\subseteq \RR$ such that $A_x \subseteq I$ for all $x\in \RR^n$. \\

We define $\fisol(A,\varepsilon)$ as the set of $(x,y)\in A$ such that every point in $A_x$ is $\varepsilon$-isolated, that is
\[
\fisol(A,\varepsilon) := \{ (x,y) \in A \ : \ A_x \subseteq \isol(A_x,\varepsilon)\}.
\]
Set
\[
\fisol(A):= \{ (x,y) \in A \ : \ A_x \subseteq \isol(A_x)\}.
\]

\begin{fact}\label{fact:opensubset}
    Let $A\subseteq \RR^{n+1}$, let $B\subseteq \RR^{n}$ be an open subset of $\pi(A)$, and let $x\in B$. Then
\[
\cl(A)_x = \cl\big((B \times \RR) \cap A\big)_x
\]
\end{fact}
\begin{proof}
Let $y\in \cl(A)_x$. Then there is a sequence $\big((x_i,y_i)\big)_{i\in \NN}$ of elements in $A$ such that $\lim_{i\to \infty} (x_i,y_i)= (x,y)$. Since $B$ is open, we may assume that $x_i \in B$ for $i\in \NN$. Thus, $y \in \cl\big((B \times \RR) \cap A\big)_x$. The other inclusion is immediate.     
\end{proof}

Let $A \subseteq \RR^m$ and let $B\subseteq \RR^n$. We say that a function $f:A\to B$ is a \textbf{local homeomorphism} if for every $x\in A$ there exists an open neighborhood $U$ around $x$ such that $f(A \cap U)$ is open in $B$ and the restriction $f|_{A\cap U}$ is a homeomorphism.\\

We say $A\subseteq \RR^n$ is \textbf{\'etale} at $x\in A$ if there is an open box $U\subseteq \RR^n$ around $x$ such that $\pi|_{A\cap U}$ is a homeomorphism onto its image in $\RR^{n-1}$.
We say that $A$ is \textbf{\'etale} if $A$ is \'etale at every $x\in A$.

\begin{fact}\label{fact:etale}
Let $A \subseteq \RR^{n+1}$ and let $B\subseteq \RR^{n}$ be an open subset of $\pi(A)$ such that $A$ is \et~at every $z\in (B\times \RR)\cap A$. Then $(B\times \RR)\cap A$ is \et.  
\end{fact}
\begin{proof}
    Let $z\in (B\times \RR)\cap A$. 
    Since $A$ is \'etale at $z$, there is an open box $U\subseteq \RR^{n+1}$ around $z$ such that $\pi|_{A\cap U}$ is a homeomorphism onto its image.
    Therefore, $\pi|_{(B\times \RR)\cap A \cap U}$ is also homeomorphic onto its image and thus $(B\times \RR)\cap A$ is \'etale at $z$.
\end{proof}

We now recall the definition of Cantor-Bendixson rank. Let $A\subseteq \RR$ be closed, and let $\lambda$ be an ordinal. We define the set $A^{[\lambda]}$ as follows:
\begin{align*}
    A^{[0]} &= A\\
    A^{[\lambda+1]} &= A^{[\lambda]}\setminus \isol(A^{[\lambda]})\\
    A^{[\lambda]} &= A\setminus \bigcup_{\mu<\lambda} \isol(A^{[\mu]}) \hspace{0.2 in} \text{if $\lambda$ is a limit.}
\end{align*}
The  \textbf{Cantor-Bendixson rank} of $A$, written $\rk(A)$, is the least ordinal $\lambda$ such that $A^{[\lambda+1]}=A^{[\lambda]}$. For $B\subseteq \RR^{n+1}$ and $x\in \pi(B)$, be aware that $B_x^{[\lambda]}$ denotes $(B_x)^{[\lambda]}$ and not $(B^{[\lambda]})_x$. We now recall some basic and well-known facts. 

\begin{fact}\label{fact:cb-homeo}\label{fact:cbmax}\label{fact:cbclosedin}
Let $A,B \subseteq \RR$ be closed. Then
\begin{enumerate}
    \item if there is a homeomorphism $f: A \to B$, then $\rk(A)=\rk(B)$,
    \item $\rk(A\cup B)\leq \max \{\rk(A),\rk(B)\}$,
    \item if $B\subseteq A$ and open in $A$, then $\rk(B)\leq \rk(A)$.
\end{enumerate}
\end{fact}
\begin{proof}
Statements (1) and (2) are well known results. For (3), observe that since $B$ is both closed and open in $A$ then $B^{[\lambda]}=A^{[\lambda]}\cap B$ for any ordinal $\lambda$.
\end{proof}

We now collect two observation from Friedman and Miller \cite{Miller-fast}.

\begin{fact}[{\cite[1.3]{Miller-fast}}]\label{fact:RK}
Let $A\subseteq \RR$ be closed and let $k\in \NN$. Then the following are equivalent:
\begin{enumerate}
    \item $A$ is countable and $\rk(A)=k$,
    \item $k$ is the least number of discrete sets whose union is $A$.
\end{enumerate}
\end{fact}

\begin{fact}[{\cite[1.4]{Miller-fast}}]\label{fact:unionsofdiscrete}
Let $A_1,\dots,A_m\subseteq \RR$, and let $k \in  \NN$ be such that $\cl(A_i)$ is a union of $k$ discrete sets for $i=1,\dots,m$. 
Then $\cl(\bigcup_{i=1}^{m} A_i)$ is a union of $k$ discrete sets.
\end{fact}

\subsection{D-minimality and noiselessness} \label{sect:noiseless} We now recall facts about d-minimal structures from \cite{Miller-tame}. Indeed, many results that we are going to use and going to prove hold in the large generality of noiseless structures.

\begin{defn} We say $\R$ is \textbf{noiseless} if every definable subset of $\RR$ either has interior or is nowhere dense.
\end{defn}

Clearly, d-minimality implies noiselessness. Noiselessness as a tameness condition is introduced and studied in \cite{Miller-tame}, although the name was only suggested later by Chris Miller. 

\begin{fact}[{\cite[Main Lemma (3,5\& 6)]{Miller-tame}}]\label{fact:sbctforn} \label{fact:fiberlemma} Suppose that $\R$ is noiseless. Let $A\subseteq \RR^{m+n}$ be definable. Then 
\begin{enumerate}
    \item every definable subset of $\RR^n$ either has interior or is nowhere dense,
    \item $\{ x \in \RR^m \ : \ \cl(A_x) \neq \cl(A)_x \}$ is nowhere dense, and
    \item $\{ x \in \RR^m \ : \ \fr(A_x) \neq \fr(A)_x \}$ is nowhere dense.
\end{enumerate}
\end{fact}

\begin{fact}[Almost continuity {\cite[Theorem 3.3]{Miller-tame}}]\label{fact:genericcnt} Suppose that $\R$ is noiseless. Let $U\subseteq \RR^{m}$ be definable, open, and nonempty, and let $f: U\to \RR$ be definable. Then there is a definable, open, dense, and nonempty set $V\subseteq U$ such that $f|_V$ is continuous. 
\end{fact}

\subsection{Sequence sets} 

Every o-minimal expansion is d-minimal. When $\R$ is o-minimal, the conclusion of \cref{thm:uniform2} follows immediately, because every finite set is discrete. Thus in order to prove \cref{thm:uniform2}, we can assume that $\R$ is not o-minimal. In this situation, an infinite set without interior is definable. Here we are interested in a special kind of such a set.   

\begin{defn}
We say $D\subseteq (0,1]$ is a \textbf{sequence set} if $D$ is an infinite discrete definable set such that $1\in D$ and $D\cup \{0\}$ is closed.
\end{defn}

Note that a sequence set is necessarily countable.

\begin{lem}\label{lem:sequenceset}
    Suppose that $\R$ is noiseless, but not o-minimal. Then $\R$ defines a sequence set.
\end{lem}
\begin{proof}
Since $\R$ is not o-minimal, there is an infinite $X\subseteq \RR$ definable in $\R$ with empty interior. Because $\R$ is noiseless, the set $X$ is nowhere dense. Thus its topological closure $\cl(X)$ is definable, infinite, and nowhere dense. Hence the open core $\R^{\circ}$ is not o-minimal. By \cite[2.8]{HM} we know that $\R$ defines an infinite, closed, and discrete set $E \subseteq \RR$. Note that $E$ is necessarily unbounded. By taking a subset and, if necessary, multiplying the elements of $E$ by $-1$, we can assume that $E \subseteq \RR_{\geq 1}$. Then 
\[
D = \{1\} \cup \{ e^{-1} \ : \ e \in E\} 
\]
is the desired sequence set.
\end{proof}

\begin{defn}
Let $D$ be a sequence set. 
We define $\lambda_D : (0,1] \to (0,1]$ as the function given by
\[
x \mapsto \max \big( (0,x] \cap D \big).
\]
\end{defn}

\subsection{Naive dimension} Let $A\subseteq \RR^n$ be non-empty.
The \textbf{dimension} of $A$, written $\dim(A)$, is the maximum $m \in \mathbb{N}\cup\{0\}$ for which there is a coordinate projection $\pi:\RR^n\to \RR^m$ such that $\pi(A)$ has interior. We set $\dim(\emptyset) := - \infty$.\\

This notion of dimension is much more convenient with regard to definability than other more classical notions of topological dimension. In particular, whenever $A\subseteq \RR^m \times \RR^n$ is definable and $k\in \NN$, then
\[
\{ x \in \RR^m \ : \dim A_x \geq k\} 
\]
is definable. For a more detailed analysis of various dimensions in this setting, see Hieronymi and Miller \cite{HM}. We gather the following facts.

\begin{fact}\label{fact:naive} Let $A,B \subseteq \RR^n$. Then
\begin{enumerate}
    \item $\dim A = n$ if and only if A has nonempty interior in $\RR^n$,
      \item $\dim A \leq \dim B$, if $A\subseteq B$,
      \item $\dim A \times B = \dim A + \dim B$.
\end{enumerate}
If $\R$ is noiseless and $A,B$ are definable, then
\begin{enumerate}
    \item[(4)] $\dim (A\cup B) = \max \{\dim A, \dim B\}$,
    \item[(5)] $\dim \cl(A) = \dim A$.
\end{enumerate}
\end{fact}
\begin{proof}
Statements (1)-(3) follow easily from the definition of the dimension. Statement (4) is \cite[Main Lemma (11)]{Miller-tame} and Statement (5) is  \cite[Main Lemma (1)]{Miller-tame}.
\end{proof}

\begin{lem}\label{lem:dim}
Suppose that $\R$ is noiseless.
Let $D$ be a sequence set and let $(A_d)_{d\in D}$ be a definable family of subsets of $\RR^n$. Then 
\[
\dim\left ( \bigcup_{d\in D} A_d \right) = \max_{d \in D}  \dim(A_d).
\]
\end{lem}

\begin{proof}
    We proceed by induction on $n$. Suppose $n=0$.
    First, consider the case that 
    \[\max_{d \in D}  \dim(A_d)=-\infty.\] 
    Then $A_d$ is empty for all $d\in D$. Thus
    \[
    \dim(\bigcup_{d\in D} A_d)=\dim \emptyset = -\infty.
    \]
    Now suppose that 
    \[\max_{d \in D}  \dim(A_d)=0.\]
    Then, for each $d\in D$, the set $A_d$ is either empty or a singleton set for each $d\in D$. Thus $\bigcup_{d\in D} A_d$ does not have interior by the Baire Category Theorem. Since there is $d\in D$ such that $A_d\neq \emptyset$, we get that $\dim(\bigcup_{d\in D} A_d)=0$.\\
    
    Now suppose that $n\in \NN_{\geq 1}$. By \cref{fact:naive}(2), it is enough to show that
    \[
\dim\left ( \bigcup_{d\in D} A_d \right) \leq \max_{d \in D}  \dim(A_d).
    \]
    If $\max_{d\in D} \dim A_d=n$, then this is trivial. So we can reduce to the case that
    \[
    \max_{d \in D} \dim (A_d)<n.
    \]
    Towards a contradiction, suppose that there is $k\in \NN$ such that 
    \[
    \dim(\bigcup_{d\in D} A_d)=k>\max_{d \in D} \dim (A_d).
    \]
    We first show that $k<n$. Let $d\in D$. Since $n \geq k > \dim A_d$, we know that $A_d$ does not have interior and hence is nowhere dense by \cref{fact:sbctforn}. Thus, its topological closure $\cl(A_d)$ is nowhere dense. By the Baire Category Theorem $\bigcup_{d\in D}\cl(A_d)$ does not have interior, and hence is nowhere dense by \cref{fact:sbctforn}. Thus
    \[
    k = \dim(\bigcup_{d\in D} A_d)\leq \dim(\bigcup_{d\in D} \cl(A_d))<n.
    \]
    Let $\pi: \RR^n \to \RR^k$ be a coordinate projection such that $\pi(\bigcup_{d\in D} A_d)$ has interior in $\RR^k$.
    Therefore 
    \[
    \dim(\bigcup_{d\in D} \pi(A_d))\geq k.
    \]
    Since $\dim(A_d)<k$ for each $d \in D$, we conclude by induction that 
    \[\dim(\bigcup_{d\in D} \pi(A_d))<k,
    \]
    a contradiction.
\end{proof}

\subsection{$\DSig$ Sets} A set $A\subseteq \RR^n$ is $\DSig$ if there is a definable family $(X_{t})_{t>0}$ of compact subsets of $\RR^n$ such that for all $s,t\in \RR_{>0}$
\begin{enumerate}
    \item $X_{s}\subseteq X_{t}$ if $s \leq t$,
    \item $A = \bigcup_{t>0} X_t$.
\end{enumerate}
Miller and Speissegger introduce $\DSig$ sets in \cite{MS-opencore}. Since it is easy to see that every $\DSig$ set is  $F_{\sigma}$, one can think of $\DSig$ sets as a definable analogue of $F_{\sigma}$ sets. Here we collect some facts about $\DSig$ sets.

\begin{fact}[{\cite[1.2,2.2]{MS-opencore}}]\label{fact:D-Sigma}
\begin{enumerate}
    \item Open and closed definable sets are $\DSig$.
    \item A finite union or finite intersection of $\DSig$ sets is $D_\Sigma$.
    \item The image of a $D_\Sigma$ set under a continuous definable function is $D_\Sigma$.
    \item Every locally closed set is $\DSig$.
\end{enumerate}
\end{fact}

We will use the following fact about dimensions of $\DSig$ sets.

\begin{fact}[{\cite[Cor. 6.7]{FHW-Compact}}]\label{fact:DSigDim}
Suppose that $\R$ does not define $\ZZ$.
    Let $A\subseteq \RR^m$ and $B \subseteq \RR^n$ be $\DSig$ and $f : A \to B$ be continuous
and definable. If there is a $d \in \NN$ such that $\dim f^{-1}(x) = d$ for all $x \in B$, then
\[
\dim f^{-1}(B) = \dim B + d.
\]
\end{fact}

Note that if $\R$ is noiseless, then it does not define $\QQ$, and hence does not define $\ZZ$.

\section{Cells and stacks}\label{section:cells-and-stacks}

In this section, we recall the definition of a cell and prove basic facts about cells. Afterwards, we introduce stacks, which are well-behaved finite unions of cells. While in this paper we are mainly interested in d-minimal structures, all results in this section hold under the weaker assumption of noiselessness. Thus, we only assume that $\R$ is noiseless.

\subsection{Cells} We first introduce the notion of a cell, defined as in an o-minimal structure. We now recall the notation used in \cite[(3.2.2)]{tametop}.  Let  $X\subseteq \RR^n$ be definable. Set
\begin{align*}
\mathcal{C}(X) &:= \{ f : X \to \RR \ : \ f \text{ is definable and continuous}\},\\
\mathcal{C}_{\infty}(X) &:= \mathcal{C} \cup \{-\infty,+\infty\},
\end{align*}
where we regard $-\infty$ and $+\infty$ as constant functions on $X$.
For $f,g \in \mathcal{C}_{\infty}(X)$, we write $f<g$ whenever $f(x)<g(x)$ for all $x \in X$. If this is the case, we set 
\[(f,g) := \{(x,y) \in X \times \RR : f(x) < y < g(x)\}.\]

\begin{defn}\label{defn:cells}
Let $n \in \NN_{\geq 1}$, and let $w=(w_1,\dots,w_n) \in \{0,1\}^n$. We now define a \textbf{$w$-cell} recursively:\newline

For $n=1$: a $0$-cell is a point in $\RR$, and a $1$-cell is an open subinterval of $\RR$.\newline
For $n>1$: a set $A\subseteq \RR^n$ is a \textbf{$w$-cell} if
\begin{enumerate}
    \item $\pi(A)$ is a $(w_1,\dots,w_{n-1})$-cell,
    \item if $w_n=0$, there is $f\in \mathcal{C}(\pi(A))$ such that $A$ is the graph of $f$,
    \item if $w_n=1$, there are $f,g\in \mathcal{C}_{\infty}(\pi(A))$ such that $f<g$ and $A=\big(f,g\big)$.
\end{enumerate}
\end{defn}

Following \cite[(3.2.6)]{tametop}, we view the one-point space $\RR^0$ as a $()$-cell, where $( )$ is the sequence of length 0. 

\begin{lem}\label{lem:celldsig}\label{lem:proj-cell}\label{cor:cell-dim}
Let $w=(w_1,\dots,w_n)\in\{0,1\}^n$, and let $A\subseteq \RR^n$ be a $(w_1,\dots,w_n)$-cell. Then 
\begin{enumerate}
    \item $A$ is connected,
    \item $A$ is locally closed, and hence $\DSig$, 
    \item $\pi_w|_A$ is a homeomorphism and $\pi_w(A)$ is open, and
    \item $\dim A = w_1+\dots+w_n$.
\end{enumerate}
\end{lem}

\begin{proof}
Statement (1) follows as in the proof of {\cite[(3.2.9)]{tametop}}.\newline

For (2), it is enough to show that $A$ is locally closed by \cref{fact:D-Sigma}. We proceed by induction on $n$. The case $n=1$ is immediate. \newline

Suppose that (2) holds for $n-1$. Let $x \in A$, and let $U\subseteq \RR^{n-1}$ be an open box such that $\pi(x)\in U$ and $\pi(A) \cap U$ is closed in $U$.\\

Suppose that $w_n=0$. Then there is $f \in \mathcal{C}(f)$ such that $A = \Gamma (f)$. Set $V:= U \times \RR$. It follows from the continuity of $f$ that $A \cap V$ is closed in  $V$.\\

Suppose that $w_n=1$. Let $f,g \in \mathcal{C}_{\infty}(\pi(A))$ be such that $f < g$ and $A= \big(f,g\big)$. 
Let $J$ be an open subinterval of $\big(f(\pi(x)),g(\pi(x))\big)$.
    Since $f,g \in \mathcal{C}_{\infty}(\pi(A))$, there is an open box $B\subseteq U$ such that for all $b\in B$ 
    \[
    J \subseteq \big(f(\pi(b)),g(\pi(b))\big).
    \]
Now set $V:=B\times J$, and observe that $A\cap V=V$, and, thus $A\cap V$ is closed in $V$. This completes the proof of (2).\\

For (3), we also proceed by induction on $n$. First, suppose that $n = 1$. If $w_1=1$, the statement follows, since $\pi_{(1)}$ is the identity. If $w_1=0$, observe that $A$ is a singleton set. Now suppose that (3) holds for $n-1$.\newline

 Suppose that $w_n=0$.
  Let $X \subseteq \RR^{n-1}$ be a $(w_1, \dots, w_{n - 1})$-cell and let $f \in\mathcal{C}(X)$ be such that $\Gamma(f) = A$. Observe that the function from $X$ to $A$ mapping $x\in X$ to $(x, f(x))$ is a homeomorphism with inverse $\pi_{(1, 1, \dots, 1, 0)}\vert_A$. Thus
\[
\pi_w = \pi_{(w_1, \dots, w_{n - 1})} \circ \pi_{(1, 1, \dots, 1, 0)}\vert_A
\]
is a homeomorphism onto an open set.\newline

Suppose that $w_{n}=1$. Let $X$ be a $(w_1, \dots, w_{n - 1})$-cell
  and let $f,g \in \mathcal{C}_{\infty}(X)$ be such that $A = \big(f, g\big)$.
  Then
\begin{align*}
\big(f, g\big) &\to \RR^{\dim X} \times \RR\\
(x, y) &\mapsto (\pi_{(w_1, \dots, w_{n - 1})}(x), y)
\end{align*}
  is a homeomorphism onto an open set by our induction hypothesis, and is equal to the projection $\pi_{(w_1, \dots, w_{n - 1}, 1)}\vert_A$.\newline

For (4), we again use induction on $n$.
Suppose that $n=1$.
There are two cases.
If $A$ is a $0$-cell, then it is a point and thus $\dim(A)=0$.
If $A$ is a $1$-cell, then it is an interval in $\RR$, and thus $A$ has interior. Therefore $\dim(A)=1$.\\

Now suppose that $n>1$.
Set 
  \[
  B \coloneqq \pi_{(1, 1, \dots, 1, 0)}(A).
  \]
By induction,
\[
\dim B = w_1 + \cdots + w_{n - 1}.
\]
First consider the case that $w_n=1$. Then for every $b \in B$
\[
\dim ((\pi_{(1, 1, \dots, 1, 0)}|_A)^{-1}(b)) = \dim A_b = 1. 
\]
By \cref{fact:DSigDim}
\[
\dim A = \dim ((\pi_{(1, 1, \dots, 1, 0)}|_A)^{-1}(B)) = \dim B + 1 = w_1+\cdots + w_n.
\]
Now consider the case that $w_n=0$. Then for every $b \in B$
\[
\dim ((\pi_{(1, 1, \dots, 1, 0)}|_A)^{-1}(b)) = \dim A_b = 0. 
\]
By \cref{fact:DSigDim}
\[
\dim A = \dim ((\pi_{(1, 1, \dots, 1, 0)}|_A)^{-1}(B)) = \dim B = w_1+\cdots + w_n.
\]
\end{proof}

So far, we only established basic properties of cells, and have not fully used our assumption of noiselessness. Indeed, only for the proof of \cref{cor:cell-dim}(4) do we need the weaker assumption that $\R$ does not define $\ZZ$. The other statements do not require any assumption on $\R$. This will change now. In the following, we prove variants of results from Section \ref{sect:noiseless}, where we replace in the statement nowhere denseness in $\RR^n$ by nowhere denseness in a cell. 

\begin{lem}[Noiselessness in cells]\label{lem:BC-cell}\label{cor:cellnowheredense}
 Let $w \in \{0,1\}^n$, let $A\subseteq \RR^n$ be a $w$-cell, and let $Z\subseteq A$. Then 
    \begin{enumerate}
        \item  $Z$ either has interior in $A$ or is nowhere dense in $A$, and 
        \item if $Z$ is nowhere dense in $A$, then $\dim Z < \dim A$.
    \end{enumerate}

\end{lem}
\begin{proof}
We first show (1). Suppose that $Z$ is somewhere dense in $A$. 
    Then $\pi_w(Z)$ is somewhere dense in $\pi_w(A)$.
    Since $\pi_w(A)$ is open by \cref{lem:proj-cell}, the projection $\pi_w(Z)$ is somewhere dense in $\RR^{\dim(A)}$.
    Thus $\pi_w(Z)$ has interior in $\RR^{\dim(A)}$ by \cref{fact:sbctforn}, and hence has interior in $\pi_w(A)$.
    Since $\pi_w|_A$ is a homeomorphism, it follows that $Z$ has interior in $A$.\newline

For (2), suppose that $Z$ is nowhere dense in $A$. Let $Y$ be the topological closure of $Z$ in $A$. Note that $Y$ is $\cl(Z)\cap A$, and hence $Y$ is $\DSig$ by \cref{lem:celldsig} and \cref{fact:D-Sigma}(2). Since $Z$ is nowhere dense in $A$, so is $Y$. Since $\pi_w|_A$ is a homeomorphism by \cref{lem:proj-cell}, we have that $\pi_w(Y)$ is nowhere dense and closed in the open set $\pi_w(A)$. Thus by \cref{fact:naive}(1) 
\[
\dim \pi_w(Y) < w_1+\dots + w_n = \dim A.
\]
Observe that $\pi_w(Y)$ is $\DSig$ by \cref{fact:D-Sigma}(5). 
By \cref{fact:DSigDim} and using that $\pi_w|_A$ is a homeomorphism,
\[
\dim Z \leq \dim Y = \dim ((\pi_w|_A)^{-1}(\pi_w(Y))) = \dim \pi_w(Y) < \dim A.
\]
\end{proof}

\begin{lem}[Almost continuity for cells]\label{cor:cellcont}
Let $w \in \{0,1\}^n$, let $A \subseteq \RR^n$ be a $w$-cell, and let $f: A \to \RR$ be definable. Then there is a definable subset $B\subseteq A$ such that $B$ is open and dense in $A$, and the restriction $f|_B$ is continuous.
\end{lem}

\begin{proof}
By \cref{lem:proj-cell}, we know that $\pi_w(A)$ is open and that  $\pi_w|_A$ is a homeomorphism. By \cref{fact:genericcnt} there is an open subset $V\subseteq \pi_w(A)$ such that $V$ is dense in $\pi_w(A)$ and $\big(f \circ (\pi_w|_A)^{-1}\big)|_V$ is continuous. Thus $(\pi_w|_A)^{-1}(V)$ is an open and dense subset of $A$ and the restriction of $f$ to $(\pi_w|_A)^{-1}(V)$ is continuous.
\end{proof}

\begin{lem}[Fiber lemma for cells]\label{lem:fl-cell}
Let $w =(w_1,\dots,w_n) \in \{0,1\}^n$, and let $A\subseteq \RR^{n+1}$ be definable such that $\pi(A)$ is a $w$-cell. Then 
    \[
    \{x \in \pi(A) : \cl(A_x)\neq \cl(A)_x\}
    \]
    is nowhere dense in $\pi(A)$.
\end{lem}

\begin{proof}
    We proceed by induction on $n$.
    If $n=0$, this is immediate. Now assume that the statement holds for all $k< n$.\newline 

    \noindent Suppose $w_1+\dots + w_n=n$. Then $\pi(A)$ is open by \cref{lem:proj-cell}. Thus the statement follows from \cref{fact:fiberlemma}.\newline

    \noindent Now consider the case that $w_1+\dots+w_n <n$. By \cref{lem:proj-cell}, the projection $\pi_w|_{\pi(A)}$ is a homeomorphism.
    Set
    \[
    B:=\{x \in \pi(A) : \cl(A_x)\neq \cl(A)_x\}
    \]
    and  \[
    C:= \{ (\pi_w(x),z)  : x \in \pi(A), (x,z)\in A\}.
    \]
    Since $\pi_w|_{\pi(A)}$ is a bijection, we have that for all $x \in \pi(A)$
\[
A_x = C_{\pi_w(x)}.
\]
Thus $\cl(A_x)=\cl(C_{\pi_w(x)})$. Since $\pi_w|_{\pi(A)}$ is a homeomorphism,
\[
\cl(A)_x = \cl(C)_{\pi_w(x)}.
\]
Hence 
\[
\pi_w(B) = \{ y \in \pi_w(\pi(A)) \ : \ \cl(C_y) \neq \cl(C)_{y} \}.
\]
Since $w_1+\dots+w_n <n$, we obtain by induction that $\pi_w(B)$ is nowhere dense in $\pi_w(\pi(A))$. Since $\pi_w|_{\pi(A)}$ is a homeomorphism, we conclude that $B$ is nowhere dense in $\pi(A)$.
\end{proof}

\subsection{Stacks} We now introduce and study stacks, which we define as well-behaved unions of finitely many cells. 

\begin{defn}\label{defn:stacks}
Let $n\in \NN_{\geq 0}$, and let $w=(w_1,\dots,w_n) \in \{0,1\}^n$. We now define a \textbf{$w$-stack} recursively:\newline
For $n=1$: a \textbf{$0$-stack} is a finite set of points in $\RR$, and a \textbf{$1$-stack} is a union of finitely many disjoint open intervals in $\RR$.\newline
For $n>1$: a set $A\subseteq \RR^n$ is a \textbf{$w$-stack} if
\begin{enumerate}
    \item $\pi(A)$ is a $(w_1,\dots,w_{n-1})$-stack,
    \item if $w_n = 0$, then for each connected component $C$ of $\pi(A)$ there are 
    \[
    f_{A,1},\dots,f_{A,m_C} \in \mathcal{C}(C)
    \]
    such that  
\begin{enumerate}
    \item $f_{A,1} < \dots < f_{A,m_C}$,
    \item  $A \cap [ C \times \RR] = \bigcup\limits_{i=1}^{m_C} \Gamma(f_{A,i})$.
\end{enumerate}
\item if $w_n = 1$, then for each connected component $C$ of $\pi(A)$ there are  
\[
f_{A,1},\dots,f_{A,m_C},g_{A,1},\dots,g_{A,m_C} \in \mathcal{C}_{\infty}(C)
\]
such that  
\begin{enumerate}
    \item $f_{A,i}<g_{A,i}$ for all $i=1,\ldots,m_C$ and either $g_{A,i}=f_{A,i+1}$ or $g_{A,i}<f_{A,i+1}$ for all $i=1,\ldots,m_C-1$, and
    \item  $A \cap [C \times \RR] = \bigcup\limits_{i=1}^{m_C} \big(f_{A,i},g_{A,_i}\big)$.
\end{enumerate}
\end{enumerate}
In this situation, we often say that $A$ is a $w$-stack \textbf{over} $\pi(A)$, and that the functions described in (2) or (3) \textbf{witness} that $A$ is a $w$-stack over $\pi(A)$.
\end{defn}

Again, we consider $\RR^0$ as the unique $()$-stack. It is immediate that a $w$-cell is a $w$-stack.

\begin{lem}\label{lem:fin-CC}
Let $w=(w_1,\dots,w_n) \in \{0,1\}^n$, and let $A\subseteq \RR^n$ be a $w$-stack. Then
\begin{enumerate}
    \item $A$ has finitely many connected components,
    \item for each connected component $C$ of $A$
        \begin{enumerate}
            \item $C$ is a $w$-cell,
            \item $C$ is clopen in $A$,
                       \item if $n>1$, then $\pi(C)$ is a connected component of $\pi(A)$ and
                       \begin{enumerate}
                           \item if
                       $w_n=0$, there are $f_{A,1},\dots,f_{A,m_{\pi(C)}} \in \mathcal{C}({\pi(C)})$
witnessing that $A$ is a $w$-stack over $\pi(A)$ and there is $j\in \{1,\dots,m_{\pi(C)}\}$ such that 
\[
C=\Gamma(f_{A,i}\big),
\]
            \item if $w_n=1$, there are $f_{A,1},\dots,f_{A,m_{\pi(C)}},g_{A,1},\dots,g_{A,m_{\pi(C)}}\linebreak \in \mathcal{C}_{\infty}({\pi(C)})$
            witnessing that $A$ is a $w$-stack over $\pi(A)$ and  there is $j\in \{1,\dots,m_{\pi(C)}\}$ such that
            \[
            C=\big(f_{A,j},g_{A,_j}\big),
            \]   
            \end{enumerate}
        \end{enumerate}
    \item $A$ is definable.
\end{enumerate}
\end{lem}
\begin{proof}
Statement (3) follows from (1) and (2). We prove the other statements by simultaneous induction on $n$.\newline

For $n=0$ and $n=1$, this follows immediately from the definition of a stack in $\RR$.  Now assume that $n>1$ and that (1)-(2) holds for $n-1$.\newline

By the inductive hypothesis, $\pi(A)$ has finitely many connected components, each of which is a $(w_1,\dots,w_{n-1})$-cell and clopen in $\pi(A)$. Thus, we may assume that $\pi(A)$ is a  $(w_1,\dots,w_{n-1})$-cell, and, in particular, connected. We split into two cases.\\

Suppose that $w_n=0$. Let $f_1,\ldots ,f_m \in \mathcal{C}(\pi(A))$ witness that $A$ is a $w$-stack.
Since $f_1<\ldots<f_m$, we know that for each $i,j\in\{1,\dots,m\}$ with $i\neq j$ 
\[
\Gamma(f_i)\cap \Gamma(f_j)=\emptyset.
\]
Moreover, since each $f_i$ is continuous and $\pi(A)$ is connected, it follows that each $\Gamma(f_i)$ is connected and clopen.
Thus the collection of graphs $\Gamma(f_1),\ldots,\Gamma(f_m)$ are the connected components of $A$. Each $\Gamma(f_i)$ is a $w$-cell by definition, and its image under $\pi$ is $\pi(A)$.\\

Suppose that $w_n=1$. Let $f_1,\ldots f_m,g_1,\ldots,g_m \in \mathcal{C}_{\infty}(\pi(A))$ witness that $A$ is a $w$-stack.
Since either $g_i < f_{i+1}$ or $g_i=f_{i+1}$ for each $i=1,\dots,m-1$, then
\[
(f_i,g_i)\cap (f_j,g_j)=\emptyset
\]
for all $i,j\in\{1,\dots,m\}$ with $i\neq j$.
Since $\pi(A)$ is connected, it follows that each $(f_i,g_i)$ is connected and clopen in $A$. Thus the sets $(f_1,g_1),\dots,(f_m,g_m)$ are the connected components of $A$. Again, each $(f_i,g_i)$ is a $w$-cell by definition, and its image under $\pi$ is $\pi(A)$.
\end{proof}

\begin{cor}\label{cor:dim-stack}\label{lem:proj-stack}
Let $w=(w_1,\dots,w_n) \in \{0,1\}^n$, and let $A\subseteq \RR^n$ be a $(w_1,\dots,w_n)$-stack. Then
\begin{enumerate}
    \item $A$  is locally closed, and hence $\DSig$,
    \item $\pi_w|_A$ is a local homeomorphism and $\pi_w(A)$ is open, and
    \item $\dim(A) = w_1+ \dots + w_n$.
\end{enumerate}
\end{cor}

\begin{proof}
For (1), it is enough to show that $A$ is locally closed. Let $a \in A$ and let $C$ be the connected component of $A$ containing $a$. By \cref{lem:fin-CC}, we know that $C$ is a $w$-cell and open in $A$. Let $U\subseteq \RR^n$ be open such that $C = A \cap U$. By \cref{lem:celldsig}(2), there is an open neighborhood $V$ of $x$ such that $C\cap V$ is closed in $V$. Thus $A\cap (U\cap V) = C\cap (U \cap V)$ is closed in $U\cap V$.\newline

For (2),  let $a \in A$ and let $C$ be the connected component of $A$ containing $a$.
      By \cref{lem:fin-CC}, we know that $C$ is a $w$-cell and there is an open box $U\subseteq \RR^n$ containing $a$ such that $A\cap U\subseteq C$.  Thus, $\pi_w|_C$ is a homeomorphism and $\pi_w(C)$ is open by \cref{lem:proj-cell}. Hence $\pi_w|_A$ is a local homeomorphism. 
      Since 
      \[
      \pi_w(A) = \bigcup_{C \text{ conn. comp. of $A$}} \pi_w(C),
      \]
   we conclude that $\pi_w(A)$ is open.\newline

For (3), observe that $A$ has finitely many connected components $C_1,\ldots,C_k$ by \cref{lem:fin-CC}, each of which is a $w$-cell.  By \cref{cor:cell-dim}, we know that $\dim(C_i)=w_1+\cdots+w_n$.
Thus, by (1) amd \cref{lem:dim}, we conclude that 
    \[\dim(A)=\dim\left(\bigcup_{i=1}^k C_i\right)=w_1+\cdots+w_n\]
as desired. Hence (3) holds.
\end{proof}

As in the case of cells, our first results about stacks have not used noiselessness. Again, only \cref{cor:dim-stack}(3) required the extra assumption of the non-definability of $\ZZ$. This changes now.

\begin{lem}[Noiselessness for stacks]\label{cor:BC-stack}\label{cor:stacknowheredense}
   Let $w \in \{0,1\}^n$, let $A\subseteq \RR^n$ be a $w$-stack and let $Z \subseteq A$ be definable. 
    Then
    \begin{enumerate}
        \item $Z$ either has interior in $A$ or is nowhere dense in $A$, and
        \item if $Z$ is nowhere dense in $A$, then $\dim Z < \dim A$.
    \end{enumerate}
  \end{lem}

\begin{proof}
  For (1), suppose that $Z$ is somewhere dense in $A$.
    Then $Z$ is somewhere dense in some connected component $C$ of $A$ by \cref{lem:fin-CC}.
    Since $C$ is a $w$-cell, we know that $Z$ has interior in $C$ by \cref{lem:BC-cell}. Thus $Z$ has interior in $A$, because $C$ is open in $A$.\newline

  For (2), suppose that $Z$ is nowhere dense in $A$. Let $C_1,\dots,C_m\subseteq A$ be the connected components of $A$, each of which is a $w$-cell. Let $i\in \{1,\dots,m\}$. By \cref{lem:fin-CC}, we have that $Z\cap C_i$ is nowhere dense in $C_i$, and thus by \cref{cor:cellnowheredense}
\[
\dim Z \cap C_i < \dim C_i = \dim A.
\]
Putting this together with \cref{fact:naive}, we conclude that
\[
\dim Z = \max_{i\in\{1,\dots,m\}} \dim(Z \cap C_i)  < \dim A.
\]
\end{proof}

\begin{lem}[Almost continuity for stacks]\label{cor:stackcont}
Let $w \in \{0,1\}^n$, let $A \subseteq \RR^n$ be a $w$-stack, and let $f: A \to \RR$ be definable. Then there is a definable subset $B\subseteq A$ such that $B$ is open and dense in $A$ and the restriction $f|_B$ is continuous.
\end{lem}
\begin{proof}
By \cref{lem:fin-CC}, there are $w$-cells $C_1,\dots,C_k$ such that these are the connected components of $A$ and each is clopen in $A$. By \cref{cor:cellcont}, for each $i=1,\dots,k$ there is a definable set $B_i\subseteq C_i$ such that $B_i$ is open and dense in $C_i$, and $f|_{B_i}$ is continuous. Since the $C_i$'s are open in $A$, it follows that $\bigcup_{i=1}^k B_i$ is open and dense in $A$, and the restriction of $f$ to $\bigcup_{i=1}^k B_i$ is continuous.
\end{proof}

\begin{lem}[Fiber lemma for stacks]\label{lem:BC-stacks}
Let $w \in \{0,1\}^n$, and let $A\subseteq \RR^{n+1}$ be definable such that $\pi(A)$ is a $w$-stack. 
  Then 
  \[
  \{x\in \pi(A) \ : \ \cl(A_x)\neq \cl(A)_x\}
  \]
  is nowhere dense in $\pi(A)$.
\end{lem}

\begin{proof}
    By Lemma \ref{lem:fin-CC}, the connected components of $\pi(A)$ are $w$-cells $C_1,\ldots, C_k\subseteq \RR^{n}$, each of which is clopen in $\pi(A)$. Thus by \cref{fact:opensubset} we have that for each $i\in \{1,\dots,k\}$ and each $x\in C_i$
    \[
    \cl(A)_x  = \cl\big((C_i \times \RR)\cap A)\big)_x \hbox{ and } \cl(A_x)  = \cl\big((C_i \times \RR)\cap A)_x\big).
    \]
    Therefore
    \begin{align*}
     \{x\in \pi(A) &: \cl(A_x)\neq \cl(A)_x \} \\
     &= \bigcup_{i=1}^k   \{x\in C_i : \cl\big((C_i \times \RR)\cap A)_x\big)\neq \cl\big((C_i \times \RR)\cap A)\big)_x \}
    \end{align*}
By Lemma \ref{lem:fl-cell} the set
    \[  \{x\in C_i : \cl\big((C_i \times \RR)\cap A)_x\big)\neq \cl\big((C_i \times \RR)\cap A)\big)_x \} \]
    is nowhere dense in $C_i$ for all $i\in\{1,\dots,k\}$.
\end{proof}

\subsection{Constructing stacks} In this subsection, we present several different ways of constructing stacks. These constructions play crucial roles in the grid decomposition theorem later. We begin with the simplest one.

\begin{lem}\label{lem:restrict-stack}
    Let $w=(w_1,\ldots,w_n) \in \{0,1\}^n$, let $A\subseteq \RR^n$ be a $w$-stack, and let $C$ be a cell such that $C\subseteq \pi(A)$. 
    Then $A\cap [C\times \RR]$ is a $w$-stack.
\end{lem}

\begin{proof}
  We will prove the statement only for $w_n=0$, as the proof for $w_n=1$ is similar.
    Since $C$ is connected, it is contained in some connected component $C'$ of $\pi(A)$. 
    Let $f_1,\ldots,f_m\in \mathcal{C}(C')$ witness that $A$ is a stack.
    Then $f_1|_C,\ldots,f_m|_C \in \mathcal{C}(C)$ witness that $A \cap [C \times \RR]$ is a stack.
\end{proof}

The argument in the proof of the following lemma is the key step in the proof of uniform finiteness in the o-minimal case.

\begin{lem}\label{lem:4.11}
Let $w\in \{0,1\}^n$ and let $A \subseteq \RR^{n+1}$ be a vertically bounded definable set such that $\pi(A)$ is a $w$-stack and for every connected component $C$ of $\pi(A)$
\begin{enumerate}
\item  $\cl(A_a)= \cl(A)_a$ for all $a\in C$,
\item  $A_a$ is finite for all $a\in C$, and
  \item $A$ is \'etale at every $(a,b)\in A\cap (C\times \RR)$.
  \end{enumerate}
  Then for every connected component $C$ of $\pi(A)$ there are  $f_1, \dots, f_{m_C} \in \mathcal{C}(C)$
  such that $f_1 < \cdots < f_{m_C}$ and 
  \[
  A\cap [C\times \RR]=\bigcup_{i = 1}^{m_C} \Gamma(f_i).
  \]
In particular, $A$ is a $(w,0)$-stack.
\end{lem}

\begin{proof}
Let $C$ be a connected component of $\pi(A)$.
  For $m \in \NN$, set
\begin{equation*}
C(m) \coloneqq \{a \in C \ : \ |A_a| = m\}.
\end{equation*}
We first prove that there is $m \in \NN$ such that $C=C(m)$. 
Since $C$ is connected, it suffices to show that $C(m)$ is open.
Pick an arbitrary point $a \in C(m)$, and let $b_1, \dots, b_m \in A_a$ be such that
$b_1 < \cdots < b_m$. Pick open intervals $J_1, \dots, J_m\subseteq \RR$ such that for
all $i \in \{1,\dots,m\}$
\begin{itemize}
    \item $b_i \in J_i$,
    \item $J_i \cap J_j = \emptyset$ for all $j \in \{1,\dots, i\}$ with $j \neq i$.
\end{itemize}
By (3), we can pick a nonempty open box $B_0 \subseteq \RR^n$ and continuous
functions $f_1, \dots, f_m \colon (B_0\cap C) \to \RR$ such that $a\in B_0$ and for all $i \in \{1,\dots,m\}$
\begin{equation}\label{eq:lem411}
A \cap (B_0 \times J_i) = \Gamma(f_i).
\end{equation}
In particular, we have that $f_i(a) = b_i$.\newline

\noindent Towards a contradiction, suppose that for every open box
$B \subseteq B_0$ containing $a$ there is $c \in B$ such that 
\[
A_c \neq \{f_1(c), \dots, f_m(c)\}.
\]
Then there is $\ell \in \{0, \dots,m\}$ and a sequence $(a_i)_{i\in \NN}$ of elements in $C$ such that $\lim_{i\to \infty} a_i= a$ and for each $i\in \NN$
\begin{equation*}
A_{a_i} \cap \Big(f_\ell(a_i), f_{\ell + 1}(a_i)\Big) \neq \emptyset. 
\end{equation*}
Here, we use the convention that $f_0(c)=-\infty$ and $f_{m + 1}(c)=\infty$ for $c\in C$.
For each $i\in \NN$ pick $u_i\in \RR$ such that 
\[
u_i\in A_{a_i} \cap \Big(f_\ell(a_i), f_{\ell + 1}(a_i)\Big).
\]
Since $A$ is vertically bounded, the sequence $((a_i,u_i))_{i\in \NN}$ is bounded. Thus, there is a convergent subsequence $((a_{i_k},u_{i_k}))_{k\in \NN}$. Let $u\in \RR$ be such that
\[
\lim_{k\to \infty } (a_{i_k},u_{i_k}) = (a,u).
\]
Since $\cl(A_a)= \cl(A)_a$ and $A_a$ is finite, we conclude that $u \in A_a$. Thus 
\[
\lim_{k\to \infty } u_{i_k} \in \{f_\ell(a), f_{\ell + 1}(a)\}.
\]
However, by \eqref{eq:lem411}, we have that for each $k\in \NN$
\[
u_{i_k} \notin (J_\ell \cup J_{\ell+1}).
\]
This contradicts that $u \in (J_\ell \cup J_{\ell+1})$.\newline

Let $m\in \NN$ be such that $C(m) = C$. For $i \in \{1,\dots,m\}$, let $h_i \colon C \to \RR$ be the function
\begin{equation*}
a \mapsto \text{($i$-th element of $A_a$)}.
\end{equation*}
Then $h_1 < \dots < h_m$ and
\begin{equation*}
\bigcup_{i = 1}^{n} \Gamma(h_i) = A.
\end{equation*}
The continuity of each $h_i$ follows from (3).
\end{proof}

The final construction in this subsection is one that is not necessary in the proof of the o-minimal cell decomposition theorem. However, it will play an important role in our proof when we need to handle open sets with infinitely many connected components. We first fix some notation.

\begin{defn}
    Let $A \subseteq \RR^n$ and let $f,g \in C_{\infty}(A)$ be such that $f<g$. The \textbf{midpoint function} $\Mid(f,g) : A \to \RR$ is the function given by
        \[
    x \mapsto \begin{cases}
    0 & f = -\infty, g = \infty\\
    f(x) + 1 & f \neq -\infty, g = \infty\\
    g(x) - 1 & f =-\infty, g \neq \infty\\
    \frac{f(x)+g(x)}{2} & \text{otherwise.}
    \end{cases}
    \]
\end{defn}

\begin{lem}\label{lem:stackmid}
Let $(w_1,\dots,w_{n}) \in \{0,1\}^{n}$, and let $A\subseteq \RR^{n+1}$ be a $(w_1,\dots,w_{n},1)$-stack. Then
\[
\Mid(A) := \{ (x,y) \in \RR^{n}\times \RR \ : \ y \in \Midp(A_x)\}
\]
is a $(w_1,\dots,w_{n},0)$-stack such that for every $C\subseteq \RR^n$, the following are equivalent:
\begin{enumerate}
    \item $C$ is a connected component of $\Mid(A)$
    \item there are $f,g \in \mathcal{C}_{\infty}(\pi(C))$ such that $(f,g)$ is a connected component of $A$ and $C=\Gamma(\Mid(f,g))$.
\end{enumerate}
\end{lem}
\begin{proof}
    Let $C$ be a connected component of $\pi(A)$, and let $f_1,\dots,f_m,g_1,\dots,g_m \in \mathcal{C}_{\infty}(C)$ witness that $A$ is a stack. For $i\in \{1,\dots,m\}$, set $h_i:=\Mid(f_i,g_i)$.
    It is easy to check that $h_1,\dots,h_m \in \mathcal{C}(C)$,
    \[
    h_1 < \dots < h_m, \hbox{ and } \Mid(A) \cap [C \times \RR] = \bigcup_{i=1}^m \Gamma(h_i).    
    \]
    Thus $\Mid(A)$ is a $(w_1,\dots,w_{n},0)$-stack. The stated equivalence follows.
\end{proof}

\subsubsection{Substacks} In the next section, we will define grids as \emph{decreasing} families of stacks. In this subsection, we define and study the partial order on stacks that makes precise what we mean by decreasing.

\begin{defn}
    Let $A,B\subseteq \RR^n$ be $w$-stacks. 
    We say $A$ is a \textbf{substack} of $B$ if every connected component of $A$ is a connected component of $B$.
\end{defn}

The sets $(0,1)$, $\RR_{>0}$, and $(0,1)\cup(1,\infty)$ are all $1$-stacks. While $(0,1)\subseteq \RR_{>0}$, it is not a substack of $\RR_{>0}$. However, the interval $(0,1)$ is a substack of $(0,1)\cup(1,\infty)$.

\begin{lem}\label{lem:projsubstack}
Let $w=(w_1,\dots,w_n)\in \{0,1\}^n$, and let $A,B\subseteq \RR^n$ be $w$-stacks such that $A$ is a substack of $B$. Then $\pi(A)$ is a substack of $\pi(B)$.    
\end{lem}
\begin{proof}
We only prove the case that $w_n=0$ as the proof in the case that $w_n=1$ is similar. Let $C$ be a connected component of $\pi(A)$ and let $f_1,\ldots,f_{m_C} \in \mathcal{C}(C)$ be such that
    \[
    A \cap [C \times \RR] = \bigcup_{i=1}^{m_C} \Gamma(f_i).
    \]
    Then each $\Gamma(f_i)$ is a connected component of $A$, and hence a connected component of $B$. Thus $C=\pi(\Gamma(f_i))$ is a connected component of $\pi(B)$ by \cref{lem:fin-CC}(2).
\end{proof}

\begin{lem}\label{lem:substack}
    Let $w=(w_1,\dots,w_n)\in \{0,1\}^n$, and let $A,B\subseteq \RR^n$ be $w$-stacks such that $\pi(A)$ is a substack of $\pi(B)$ and $w_n=0$. If $A\subseteq B$, then $A$ is a substack of $B$.
\end{lem}
\begin{proof}
Let $C\subseteq \RR^n$ be a connected component of $A$. Since $A$ is a $w$-stack, there is $f \in \mathcal{C}(\pi(C))$ such that $\pi(C)$ is a connected component of $\pi(A)$ and $C = \Gamma(f)$. Hence $\pi(C)$ a connected component of $\pi(B)$. Since $B$ is a $w$-stack, there are $f_1,\dots,f_m \in \mathcal{C}(C)$ such that 
\[
B \cap (C\times \RR) = \bigcup_{i=1}^m \Gamma(f_i).
\]
Let $i\in \{1,\dots,m\}$ be such that $\Gamma(f_i)\cap \Gamma(f)\neq \emptyset$. Since $\Gamma(f_i)$ is a connected component of $B$ and $A\subseteq B$, we get that $\Gamma(f)\subseteq \Gamma(f_i)$. Since the domains of $f$ and $f_i$ are the same the functions must be the same. Thus $C$ is a connected component of $B$.
\end{proof}

\begin{lem}\label{lem:substack2}
    Let $w=(w_1,\dots,w_{n+1})\in \{0,1\}^{n+1}$, let $A,B \subseteq \RR^{n+1}$ be $w$-stacks, and let $A',B'\subseteq \RR^{n}$ be $w'$-stacks such that 
    \begin{enumerate}
        \item $A$ is a substack of $B$,
        \item $A'$ is a substack of $B'$,
        \item $\pi(A)\cap A'$ is a substack of $A'$, and
        \item $A\cap [A' \times \RR]$ is a stack.
    \end{enumerate}
    Then $A\cap [A' \times \RR]$  is a substack of $B\cap [B'\times \RR]$.
\end{lem}

\begin{proof}
    We will prove this only for $w_{n+1}=0$ as the proof for $w_{n+1}=1$ is similar. Let $C_0$ be a connected of $A\cap [A' \times \RR]$. By (4) there is a connected component $C'$ of $A'\cap \pi(A)$ and $f_0\in \mathcal{C}(C')$ such that $C_0=\Gamma(f_0)$. Let $C$ be the connected component of $\pi(A)$ containing $C'$. Then there is $f\in \mathcal{C}(C)$ such that $f|_{C'} = f_0$ and $\Gamma(f)$ is a connected component of $A$. By (1), we know that $\Gamma(f)$ is a connected component of $B$ as well.
    Since $C'$ is a connected component of $A'\cap \pi(A)$, it is a connected component of $A'$ by (3), and thus a connected component of $B'$ by (2).
    Hence
    \[ 
    C_0 \subseteq \Gamma(f_0|_{C'}) \subseteq B \cap [B' \times \RR].
    \]
    Towards a contradiction, suppose that $C_0$ is not a connected component of $B \cap [B'\times \RR]$.
    Then there is $C''\subseteq \RR^{n}$ containing $C'$ such that $C''\neq C'$ and $\Gamma(f|_{C''})$ is a connected component of $B \cap [B' \times \RR]$.
    However, this contradicts that $C'$ is a connected component of $B'$.
\end{proof}

\begin{lem}\label{lem:substack-CC}
    Let $w=(w_1,\dots,w_n)\in \{0,1\}^n$, let $A\subseteq \RR^n$, and let $B\subseteq \RR^n$ be a $w$-stack such that
    $A\subseteq B$ and every connected component of $A$ is a connected component of $B$. Then $A$ is a stack and a substack of $B$.
\end{lem}

\begin{proof}
    We need only prove that $A$ is a stack, because then it follows that $A$ is a substack of $B$.
    We proceed by induction on $n$. The case $n=0$ is immediate.\newline

    Let $n\geq 1$, and suppose that the statement holds for $n-1$.
    We only consider the case when $w_n=0$, since the proof for $w_n=1$ is similar.
    Let $C$ be a connected component of $\pi(A)$. 
    Then there is a connected component $C'$ of $\pi(B)$ such that $C\subseteq C'$.
    Because $B$ is a $w$-stack, there are $f_1,\ldots,f_{m_{C'}} \in \mathcal{C}(C')$ such that 
    \[ 
    B\cap[C'\times \RR] = \bigcup_{i=1}^{m_{C'}} \Gamma(f_i).
    \]
    Since $A\subseteq B$, 
    \[
    A\cap [C\times \RR] \subseteq B \cap [C'\times \RR].
    \]
    Because each connected component of $A$ is a connected component of $B$, there is $I\subseteq \{1,\ldots,m_ {C'}\}$ such that
    \[ A\cap [C\times \RR] = \bigcup_{i\in I} \Gamma(f_i|_{C}).\]
    If $C\neq C'$, then $\Gamma(f_i|_{C})$ is a connected component of $A$ for $i\in I$, but not a connected component of $B$, a contradiction.
    Thus $C=C'$. 
    By induction, we conclude that $\pi(A)$ is a stack and a substack of $\pi(B)$. 
    Then the functions $f_i$ with $i\in I$ witness that $A$ is a stack.
\end{proof}

\begin{lem}\label{lem:midsubstack}
Let $(w_1,\dots,w_{n}) \in \{0,1\}^{n}$ and let $A, B\subseteq \RR^{n+1}$ be $(w_1,\dots,w_{n},1)$-stacks such that $A$ is a substack of $B$. Then $\Mid(A)$ is a substack of $\Mid(B)$.
\end{lem}
\begin{proof}
Let $d,e\in D$ with $d>e$ and let $C$ be a connected component of $\Mid(A)$. By \cref{lem:stackmid} there are $f,g \in \mathcal{C}_{\infty}(\pi(C))$ such that  $(f,g)$ is a connected component of $A$ and 
\[
C = \Gamma(\Mid(f,g)).
\]
Then $(f,g)$ is a connected component of $B$, because $A$ is substack of $B$. Thus $\Gamma(\Mid(f,g))$ is a connected component of $\Mid(B)$ by \cref{lem:stackmid}.
\end{proof}

\section{Grids}\label{section:grids}
Both cells and stacks have finitely many connected components, and hence might live in an o-minimal structure. We now finally introduce grids, which will help us study sets with infinitely many connected components. 

\begin{defn}\label{defn:grid}
Let $n\in \NN_{\geq 1}$, let $w=(w_1,\dots,w_n) \in \{0,1\}^n$, and let $D\subseteq (0,1]$ be a sequence set. Let $\mathcal{S}$ be a definable family $(S_{d})_{d \in D}$ of $w$-stacks. 
We say $\mathcal{S}$ is a \textbf{$(w,D$)-grid} if 
\begin{enumerate}
    \item $(\pi(S_d))_{d \in D}$ is a $((w_1,\dots,w_{n-1}),D)$-grid if $n>1$,
    \item $S_d$ is a substack of $S_e$ for $d,e\in D$ with $d>e$,
     \item $\bigcup_{d \in D} S_d$ is étale if $w_n=0$, 
\end{enumerate}
For $n=0$, the unique $((),D)$-grid is the constant family $(\RR^0)_{d\in D}$. \\

We say $\bigcup_{d \in D} S_d$ is the \textbf{surface} of $\mathcal{S}$, and write $\surf(\mathcal{S})$. We write $\pi(\es)$ for the $((w_1,\dots,w_{n-1}),D)$-grid $(\pi(S_d))_{d \in D}$. 
\end{defn}

We often suppress the $D$ in $(w,D)$-grid if $D$ can be deduced from the context, and simply say that $\mathcal{S}$ is a $w$-grid. Indeed, as we show now, we can always change the sequence set without changing the surface.

\begin{defn} Let $D_1$ and $D_2$ be sequence sets, and let $\es$ be a $(w,D_1)$-grid. We define the \textbf{$D_2$-restriction} of $\es$, written  $\es|_{D_2}$, as the definable family $(S'_{d})_{d\in D_2}$ such that for every $d\in D_2$
\[
S'_{d} := S_{\lambda_{D_1}(d)}.
\]
\end{defn}

\begin{lem}\label{lem:changeseqs}
Let $w=(w_1,\dots,w_n) \in \{0,1\}^n$, let $D_1$ and $D_2$ be sequence sets, and let $\es$ be a $(w,D_1)$-grid. Then $\es|_{D_2}$ is a $(w,D_2)$-grid and
    \[
    \surf(\es) = \surf(\es|_{D_2}).
    \]
\end{lem}

\begin{proof}
    We first show that 
    \[\surf(\es)=\bigcup_{d\in D_2} S_{\lambda_{D_1}(d)}.\]
    The inclusion $\bigcup_{d\in D_2} S_{\lambda_{D_1}(d)} \subseteq \surf(\es)$ is immediate from the definitions. Thus, we only need to prove that $\surf(\es)\subseteq \bigcup_{d\in D_2} S_{\lambda_{D_1}(d)}$.
    Let $d\in D_1$.
    Since $D_1$ and $D_2$ have limit points at $0$, there is $e\in D_2$ with $e<d$.
    But then $\lambda_{D_1}(e)\leq d$ as well, and therefore $S_d\subseteq S_{\lambda_{D_1}(e)}$.\\

Let $d,e\in D_2$ be such that $e<d$. We show that $S_{\lambda_{D_1}(d)}$ is a substack of $S_{\lambda_{D_1}(e)}$. Let $C$ be a connected component of $S_{\lambda_{D_1}(d)}$.
Since $e<d$, we have that $\lambda_{D_1}(e)\leq \lambda_{D_1}(d)$. Thus $C$ is a connected component of $S_{\lambda_{D_1}(e)}$, because $\es$ is a $(w,D_1)$-grid.\\
  
We now prove by induction on $n$ that $\es|_{D_2}$ is a $(w,D_2)$-grid. The case $n=0$ is trivial. Now suppose that $n>1$ and the claim holds for $n-1$. Since
\[
\pi(\es|_{D_2}) = \pi(\es)|_{D_2},
\]
we get from the induction hypothesis that $\pi(\es|_{D_2})$ is a $(w_1,\ldots,w_{n-1},D_2)$-grid.
    If $w_n=0$,
     then $\bigcup_{d\in D_2} S_{\lambda_{D_1}(d)}$ is étale, because $\bigcup_{d\in D} S_d$ is.
    Thus, together with the above observations, we conclude that $\es|_{D_2}$ is a $(w,D_2)$-grid and $\surf(\es) = \surf(\es|_{D_2}).$
\end{proof}

Fix now a sequence set $D$ for the remainder of this section.

\begin{lem}\label{lem:CC0-grid}
Let $w=(w_1,\dots,w_n) \in \{0,1\}^n$, let $\mathcal{S}=(S_d)_{d\in D}$ be a $w$-grid, and let $C\subseteq \RR^n$. Then 
\begin{enumerate}
    \item $C$ is a connected component of $\surf(\mathcal{S})$ if and only if there is $d\in D$ such that $C$ is a connected component of $S_d$,
    \item each connected component of $\surf(\mathcal{S})$ is clopen in $\surf(\mathcal{S})$, and
    \item $S_d$ is clopen in $\surf(\es)$ for each $d\in D$.
\end{enumerate}
\end{lem}

\begin{proof}
Since $S_d$ has finitely many connected components for each $d\in D$ by \cref{lem:fin-CC}, the third statement follows from the second. Set $A:=\surf(\mathcal{S})$. We prove (1) and (2) by induction on $n$. The base case that $n=0$ is trivial.\newline

    Now suppose the statement holds for $n-1$. 
    Let $C$ be a connected component of $A$. Then there is a connected component $C'$ of $\pi(A)$ such that $\pi(C)\subseteq C'$.
    By our inductive assumption, there is $d\in D$ such that $C'$ is a connected component of $\pi(S_d)$. Let $x\in C$ and let $e \in D_{\leq d}$ be maximal such that $x \in S_e$. We will show that $C$ is a connected component of $S_e$. Since $\pi(S_d)$ is a substack of $\pi(S_e)$, we have that $C'$ is a connected component of $\pi(S_e)$.\\

    Suppose first that $w_n=0$.  Let $f_1,\ldots, f_m \in \mathcal{C}(C')$ witness that $S_e$ is a $w$-stack.
    Then 
    \[
    S_e\cap [C' \times \RR]=\bigcup_{i=1}^m \Gamma(f_i).
    \]
    Let $j\in \{1,\dots,m\}$ be such that $x \in \Gamma(f_j).$ Since $\Gamma(f_j)$ is connected and $x\in C$, we have that $\Gamma(f_j)\subseteq C$. Thus, it is enough to show that $\Gamma(f_j)$ is a connected component of $A$.
    Because $\Gamma(f_j)$ is connected, it is enough to show that $\Gamma(f_j)$ is maximal with this property.  
    Since $C'$ is closed in $\pi(A)$, it follows that $\Gamma(f_j)$ is closed in $A$. Let $y \in \Gamma(f_j)$.
    Since $A$ is étale at $y$, there is an open box $U\subseteq \RR^n$ containing $x$ such that $\pi|_{U\cap A}$ is a homeomorphism onto its image. Thus, 
    \[
    U\cap A = U\cap \Gamma(f_j).
    \]
    Hence $\Gamma(f_j)$ is clopen in $A$, and hence a connected component of $A$.\newline

    Now suppose that $w_n=1$. Let $f_1,\ldots,f_m,g_1,\ldots,g_m \in \mathcal{C}_{\infty}(C')$ witness that $S_e$ is a $w$-stack. 
    Then 
    \[
    S_e\cap [C' \times \RR] = \bigcup_{i=1}^m \big(f_i,g_i\big).
    \]
    Let $j\in \{1,\dots,m\}$ be such that $x \in \big(f_j,g_j\big).$ 
    Since $(f_j,g_j)$ is connected and $x\in C$ we have that $(f_j,g_j)\subseteq C$. Thus, it is enough to show that $(f_j,g_j)$ is a connected component of $A$. Because $(f_j,g_j)$ is connected, it is enough to show that $(f_j,g_j)$ is maximal with this property. Since $C'$ is open in $\pi(A)$, we have that $(f_j,g_j)$ is open in $A$. Let $y \in A \setminus (f_j,g_j)$, and let $e' \in D_{\geq e}$ be such that $y \in S_{e'}$. Since $S_{e'}$ is a $w$-stack, there are definable continuous functions $f,g : C'\to \RR$ such that $y \in (f,g)$. Since $S_e$ is a subgrid of $S_{e'}$ and $y\notin (f_j,g_j)$, we have that $(f_j,g_j) \cap (f,g) =\emptyset$. Since $C'$ is closed in $\pi(A)$, it follows that $(f_j,g_j)$ is closed in $A$. Hence $(f_j,g_j)$ is clopen in $A$, and hence a connected component of $A$.\\

    We have established (2) and the forward direction of (1). We now prove the backwards direction of (1). Let $d\in D$, and let
    Let $C$ be a connected component of $S_d$. Let $C'$ be a connected subset of $A$ such that $C\subseteq C'$.\newline

    Suppose first that $w_n=0$.
    Then there is $f\in \mathcal{C}(\pi(C))$ such that $C=\Gamma(f)$  and $\pi(C)$ is a connected component of $\pi(S_d)$. 
    Then $\pi(C)\subseteq\pi(C')$, and thus, by our inductive hypothesis, $\pi(C')=\pi(C)$.
    By the forward direction of $(1)$, there is $e\in D$ such that $C'$ is a connected component of $S_e$. Thus, there is a continuous definable function $g:\pi(C)\to \RR$ such that $C'=\Gamma(g)$.
    But then $f=g$, and so $C=C'$. Hence, $C$ is a connected component of $A$.\\

    Suppose now that $w_n=1$. Let $f,g \in \mathcal{C}(\pi(C))$ be such that $C=\big(f,g\big)$ and $\pi(C)$ is a connected component of $\pi(S_d)$.
    Similarly to above, we conclude that $\pi(C')=\pi(C)$.
    By the forward direction of $(1)$ there is $e\in D$ such that $C'$ is a connected component of $S_e$.
    Thus, there are $\hat{f},\hat{g}\in\mathcal{C}(\pi(C))$ such that $C'=\big(\hat{f},\hat{g}\big)$. Since $S_d$ is a substack of $S_e$, we conclude that $f=\hat{f}$ and $g=\hat{g}$.
    Thus, $C=C'$ and $C$ is a connected component of $A$.
\end{proof}

\begin{cor}\label{cor:thingrid}
   Let $\mathcal{S}=(S_d)_{d\in D}$ be a $(0,\dots,0)$-grid in $\RR^n$. Then $\surf(\es)$ is discrete.
\end{cor}
\begin{proof}
We proceed by induction on $n$. The case $n=0$ is trivial. Now suppose that $n\geq 1$ and that the statement holds for $n-1$.\newline

Let $x\in\surf(\es)$, and let $d\in D$ be such that $x\in S_d$.
By our induction hypothesis, $\{\pi(x)\}$ is a connected component of $\pi(S_d)$.
Thus $\{x\}$ is a connected component of $S_d$.
So $\{x\}$ is a connected component of $\surf(\es)$ as well and is clopen in $\surf(\es)$ by \cref{lem:CC0-grid}.
Therefore, $\surf(\es)$ is discrete.
\end{proof}

\begin{lem}\label{lem:proj-grid}\label{cor:grid-dim}
Let $w \in \{0,1\}^n$, let $\mathcal{S}=(S_d)_{d\in D}$ be a $w$-grid. Then 
\begin{enumerate}
    \item $\surf(\es)$ is locally closed, and hence $\DSig$,
    \item $\pi_w|_{\surf(\mathcal{S})}$ is a local homeomorphism and $\pi_w(\surf(\mathcal{S}))$ is open, and
    \item   $\dim(\surf(\mathcal{S})) = w_1+\dots + w_n$.
\end{enumerate}
\end{lem}

\begin{proof}
Set $A:= \surf(\mathcal{S})$.    
For (1), it is enough to show that $A$ is locally closed. Let $a \in A$ and let $C$ be the connected component of $A$ containing $a$. By \cref{lem:fin-CC} and \cref{lem:CC0-grid}, we know that $C$ is a $w$-cell and clopen in $A$. Let $U\subseteq \RR^n$ be open such that $C = A \cap U$. By \cref{lem:celldsig}(2), there is an open neighborhood $V$ of $x$ such that $C\cap V$ is closed in $V$. Thus $A\cap (U\cap V) = C\cap (U \cap V)$ is closed in $U\cap V$.\newline

For (2), let $a\in A$ and let $C$ be the connected component of $A$ containing $a$.
    By Lemma \ref{lem:CC0-grid} we know that $C$ is clopen in $A$ and there is $d\in D$ such that $C$ is a connected component of $S_d$. Let $U\subseteq \RR^n$ be an open ball containing $a$ such that $A\cap U\subseteq C\cap U$. Since $C\subseteq S_d$,
    \[
    S_d \cap U = A \cap U = C\cap U.
    \]
    By Lemma \ref{lem:proj-stack} we get that $\pi_w|_{S_d\cap U}$ is a local homeomorphism and $\pi_w(S_d)$ is open. Therefore $\pi_w|_A$ is a local homeomorphism.
    Since
    \[
    \pi_w(A) = \bigcup_{d \in D} \pi_w(S_d), 
    \]
    we conclude that $\pi_w(A)$ is open.\newline

By \cref{lem:dim}
    \[\dim(A)=\dim\left(\bigcup_{d\in D} S_d\right) = \max_{d\in D}\dim(S_d)\]
    and $\dim(S_d)=w_1+\cdots+w_n$ by \cref{cor:dim-stack}. Thus (3) holds.
\end{proof}

As in the cases of cells and stacks, we have established basic properties of grids and surfaces of grids without invoking noiseless. Only \cref{cor:grid-dim}(3) used the assumption that $\R$ does not define $\ZZ$. Analogous to the previous section, the first real use of noiseless is the following lemma.

    \begin{lem}[Noiselessness for grids] \label{lem:BC-grids}\label{cor:gridnowheredense}\label{cor:boundary}
Let $\mathcal{S}=(S_d)_{d\in D}$ be a grid, and let $Z\subseteq \surf(\mathcal{S})$ be definable. Then 
    \begin{enumerate}
        \item $Z$ either has interior in $\surf(\mathcal{S})$ or is nowhere dense in $\surf(\mathcal{S})$,
        \item if $Z$ is nowhere dense in $\surf(\mathcal{S})$, then $\dim Z < \dim \surf(\mathcal{S})$,
        \item the boundary of $Z$ in $\surf(\mathcal{S})$ is nowhere dense in $\surf(\mathcal{S})$.
    \end{enumerate}
\end{lem}

\begin{proof}
Set $A:=\surf(\mathcal{S})$. For (1), suppose that $Z$ is somewhere dense in $A$.
    By \cref{lem:CC0-grid}(2) we can assume that $Z$ is dense in some connected component $C$ of $A$.
    Then, by \cref{lem:CC0-grid}(1), there is $d\in D$ such that $C$ is a connected component of $S_d$.
    Thus $Z$ is somewhere dense in $S_d$, and has interior in $S_d$  by \cref{lem:BC-stacks}.
    Therefore $Z$ has interior in $A$ by \cref{lem:CC0-grid}(2).\newline

For (2), suppose that $Z$ is nowhere dense in $A$. By \cref{lem:CC0-grid}, we know that $Z\cap S_d$ is nowhere dense in $S_d$. Thus, by \cref{cor:stacknowheredense}
\[
\dim (Z \cap S_d) < \dim S_d \leq \dim A.
\] 
By \cref{lem:dim}
\[
\dim Z = \dim \bigcup_{d\in D} (Z\cap S_d) = \max_{d\in D} \dim (Z\cap S_d) < \dim A.
\]    

For (3), note that the boundary of $Z$ in $\surf(\mathcal{S})$ does not have interior in $\surf(\mathcal{S})$, and hence is nowhere dense in $\surf(\mathcal{S})$ by (1).
\end{proof}      

\begin{lem}[Almost continuity for grids]\label{cor:gridcont}
Let $w \in \{0,1\}^n$, let $\es=(S_d)_{d\in D}$ be a $w$-grid, and let $f: \surf(\es) \to \RR$ be definable. Then there is a definable subset $B\subseteq \surf(\es)$ such that $B$ is open and dense in $\surf(\es)$ and the restriction $f|_B$ is continuous.
\end{lem}
\begin{proof}
Let $Y\subseteq \surf(\es)$ be the set of points at which $f$ is not continuous. Let $Z$ be the topological closure of $Y$ in $\surf(\es)$. Set $B:= \surf(\es)\setminus Z$. By definition, $B$ is open and $f|_B$ is continuous. It is left to show that $B$ is dense in $\surf(\es)$. Let $U\subseteq \RR^n$ be an open box such that $\surf(\es)\cap U \neq\emptyset$. By \cref{lem:CC0-grid}, we may shrink $U$ such that there is $d\in D$ with 
\begin{equation}\label{eq:gridcont}
U\cap \surf(\es) = U \cap S_d.    
\end{equation}
By \cref{cor:stackcont} there is $B'\subseteq S_d$ such that $B'$ is open and dense in $S_d$ and the restriction $f|_{B'}$ is continuous. From \eqref{eq:gridcont}, it follows that there is $x \in U \cap \surf(\es)$ at which $f$ is continuous.
\end{proof}

\begin{lem}\label{lem:grid-inter}
Let $\mathcal{S}$ be a grid, and let $Z_1,\dots,Z_m\subseteq \surf(\mathcal{S})$ be definable such that $\bigcup_{i=1}^m Z_i$ has interior in $\surf(\es)$.
    Then there is $i\in\{1,\dots,m\}$ such that $Z_i$ has interior in $\surf(\es)$.
\end{lem}

\begin{proof}
    Towards a contradiction, suppose that $Z_i$ does not have interior in $\surf(\es)$ for $i=1,\dots,m$. By \cref{lem:BC-grids}, each $Z_i$ is nowhere dense in $\surf(\es)$.
    Then $\bigcup_{i=1}^k Z_i$ is nowhere dense in $\surf(\es)$, contradicting our assumption.
\end{proof}

\begin{lem}[Fiber lemma for grids]\label{cor:FL-grids}
    Let $A\subseteq\RR^{n+1}$ be such that $\pi(A)$ is a surface of a grid. 
    Then  $\{x\in \pi(A): \cl(A_x)\neq \cl(A)_x \}$ is nowhere dense in $\pi(A)$.
\end{lem}

\begin{proof}
    Let $\mathcal{S}=(S_d)_{d\in D}$ be a grid such that $\surf(\mathcal{S})=\pi(A)$.
    Set 
    \[
    B:=\{x\in \pi(A): \cl(A_x)\neq \cl(A)_x \}.
    \]
    By \cref{lem:BC-grids} either $B$ has interior or is nowhere dense in $\pi(A)$.
    Suppose that $B$ has interior in $\pi(A)$.
    Then, by \cref{lem:CC0-grid}(2), there is a connected component $C$ of $\pi(A)$ such that $B$ has interior in $C$. Since $C$ is clopen,
    we have that for each $x\in C$ 
    \[
    \cl(A)_x  = \cl\big((C \times \RR)\cap A)\big)_x \hbox{ and } \cl(A_x)  = \cl\big((C \times \RR)\cap A)_x\big).
    \]
 Therefore
    \begin{align*}
     B\cap C &= \{x\in C : \cl(A_x)\neq \cl(A)_x \} \\
     &=   \{x\in C : \cl\big((C \times \RR)\cap A)_x\big)\neq \cl\big((C \times \RR)\cap A)\big)_x \}.
    \end{align*}
    By \cref{lem:CC0-grid}(1) there is $d\in D$ such that $C$ is a connected component of $S_d$. Thus $C$ is a cell by \cref{lem:fin-CC}.
By Lemma \ref{lem:fl-cell} 
    \[  \{x\in C : \cl\big((C \times \RR)\cap A)_x\big)\neq \cl\big((C \times \RR)\cap A)\big)_x \} \]
    is nowhere dense in $C$. This contradicts that $B$ has interior in $C$.
\end{proof}

\subsection{Constructing grids} Moving forward as in the previous section, we describe several different ways
of constructing grids. Of course, we use these constructions in the proof of the grid decomposition theorem later.

\begin{lem}\label{lem:grid-above}
Let $w \in \{0,1\}^n$, and let $\es=(S_d)_{d\in D}$ be a $w$-grid in $\RR^n$. 
    Then $(S_d \times \RR)_{d\in D}$ is a $(w,1)$-grid in $\RR^{n+1}$.
\end{lem}

\begin{proof}
    For each $d\in D$, define $S_d':= S_d \times \RR$, and set $\es' := (S_d')_{d\in D}$. Clearly, $\pi(S_d')_{d\in D}$ is
 just $(S_d)_{d\in D}$, and hence a $w$-grid. It is also easy to see that $S'_d$ is a $(w,1)$-stack for every $d \in D$.
    Let $d,e\in D$ be such that $d>e$.
    Then every  connected component of $S_d'$ is of the form $C\times \RR$, where $C$ is a connected component of $S_d$. Since $S_d$ is a substack of $S_e$, the set $C$ is a connected component of $S_e$ as well.
    Thus $C\times \RR$ is a connected component of $S_e'$. Hence $S_d'$ is a substack of $S'_e$. Therefore $(S_d')_{d\in D}$ is a $(w,1)$-grid.
\end{proof}

The next lemma is indeed another key step in the proof the grid decomposition theorem, roughly replacing the argument from \cite[(3.2.15)]{tametop} in the o-minimal case.

\begin{lem}\label{lem:grid-complement}\label{cor:grid-complement}
Let $w\in \{0,1\}^{n}$, and let $A\subseteq \RR^{n+1}$ be a vertically bounded definable set such that $\pi(A)$  is the surface of a $w$-grid and $A_x$ is open for all $x\in \pi(A)$, and let $\es_1=(S_{1,d})_{d \in D},\dots,\es_m=(S_{m,d})_{d \in D}$ be $(w,0)$-grids over $\pi(A)$ such that their surfaces are disjoint and 
    \begin{align*}
    B_1&:=\{ (x,y)  \in C \times \RR \ : \ y \in \operatorname{bd}(A_x)\},\\
    B_2&:=\{ (x,y)  \in C \times \RR \ : \ y \in \Midp(A_x)\}
    \end{align*}
    can be written as unions of these surfaces. Then $A$ is the surface of a $(w,1)$-grid.
\end{lem}

\begin{proof}
Let $\es = (S_d)_{d \in D}$ be the $w$-grid such that $\pi(A) = \surf(\es)$.
Let 
    \begin{align*}
        \lambda : A &\to \RR\\
        (x,y) &\mapsto \max \Big( (\RR\setminus A_x) \cap (-\infty,y) \Big)
    \end{align*}
    and let
    \begin{align*}
        \mu : A &\to \RR\\
        (x,y) &\mapsto \min \Big( (\RR\setminus A_x) \cap (y,\infty)\Big).
    \end{align*}
Since $(x,y) \in A$, then
\begin{equation}\label{eq:lemgrco3}
\big(\lambda(x,y),\mu(x,y)\big) \subseteq A_x.
\end{equation}
Set
\[
S^*_{d} := \{ (x,y) \in (S_d \times \RR) \cap A \ : \ \bigvee_{i,j \in \{1,\dots,m\}} \big( (x,\lambda(x,y)) \in S_{i,d} \wedge (x,\mu(x,y)) \in S_{j,d}\big)\}.
\]
Since $A_x$ is open for every $x \in \pi(A)$, it is easy to check that
\[
A = \bigcup_{d\in D} S^*_d.
\]
Thus, it is left to show that $(S^*_d)_{d\in D}$ is a $(w,1)$-grid.\\

Let $C\subseteq \RR^{n}$ be a connected component of $S_d$. For $i\in \{1,\dots,m\}$ and $d\in D$, let $f_{i,d,1},\dots,f_{i,d,m_{i,d}} \in \mathcal{C}(C)$ be such that
\begin{enumerate}
    \item $f_{i,d,1}<\cdots<f_{i,d,m_{i,d}}$,
    \item $S_{i,d}\cap (C\times \RR) = \bigcup_{j=1}^{m_{i,d}} \Gamma(f_{i,d,j})$.
\end{enumerate}
Since the surfaces of $\es_1,\dots,\es_m$ are disjoint, 
\[
\Gamma(f_{i,d,k}) \cap  \Gamma(f_{j,d,\ell}) = \emptyset,
\]
for all $i,j \in \{1,\dots,m\}$ and $k,\ell\in \NN$ whenever $i\neq j$ or $k\neq \ell$. Since $C$ is connected, we get in such a situation that
\begin{equation}\label{eq:lemgrco1}
\hbox{either } f_{i,d,k} < f_{j,d,\ell} \hbox{ or } f_{j,d,\ell} < f_{i,d,k}.
\end{equation}

Let $x_0\in C$. Since $S_{1,d},\dots,S_{m,d}$ are $(w,0)$-stacks, there are $y_1,\dots,y_{2p} \in \RR$ such that
\begin{equation}\label{eq:lemgrco2}
y_1<\dots < y_{2p}, \text{ and } (S_d^*)_{x_0} = \bigcup_{i=0}^{p-1} \big(y_{2i+1},y_{2i+2}\big).
\end{equation}
Let $k_1,\dots,k_{2p},\ell_1,\dots,\ell_{2p} \in \NN$ be such that for $i=1,\dots,2p$
\[
y_i = f_{k_i,d,\ell_i}(x).
\]
By \eqref{eq:lemgrco1},
\[
f_{k_1,d,\ell_1} < \dots < f_{k_{2p},d,\ell_{2p}}.
\]
To show that $S_d^*$ is a $(w,1)$-stack, it is enough to show that
\[
S_d^*\cap (C\times \RR) = \bigcup_{i=0}^{p-1} \big(f_{k_{2i+1},d,\ell_{2i+1}}, f_{k_{2i+2},d,\ell_{2i+2}}).
\]

Let $i\in \{0,\dots,p-1\}$. First we show that 
\[
\big(f_{k_{2i+1},d,\ell_{2i+1}},f_{k_{2i+2},d,\ell_{2i+2}}\big) \subseteq S_d^*\cap (C\times \RR) .
\]
Let $(x,y) \in  \big(f_{k_{2i+1},d,\ell_{2i+1}},f_{k_{2i+2},d,\ell_{2i+2}}\big)$. It is enough to show that
\[
\lambda(x,y) = f_{k_{2i+1},d,\ell_{2i+1}}(x) \hbox{ and } \mu(x,y) = f_{k_{2i+2},d,\ell_{2i+2}}(x).
\]
Suppose not. Then there is $u \in \{1,\dots,m\}$ and $v \in \{1,\dots,m_{u,e}\}$ such that $\surf(\es_u)\subseteq B_1$ and
\[
f_{k_{2i+1},d,\ell_{2i+1}}(x) < f_{u,d,v}(x) < f_{k_{2i+2},d,\ell_{2i+2}}(x).
\]
By \eqref{eq:lemgrco1}, 
\[
f_{k_{2i+1},d,\ell_{2i+1}}< f_{u,d,v} <f_{k_{2i+2},d,\ell_{2i+2}}.
\]
But then $f_{u,d,v}(x_0) \in (y_{2i+1},y_{2i+2})$, contradicting \eqref{eq:lemgrco2}.\\

Let $(x,y) \in S_d^*\cap (C\times \RR) $. Let $u_1,u_2 \in \{1,\dots,m\}$ and let $v_1,v_2 \in \NN$ be such that
\[
\lambda(x,y) = f_{u_1,d,v_1}(x) \hbox{ and } \mu(x,y) = f_{u_2,d,v_2}(x).
\]
It suffices to show that
\[
\big(f_{u_1,d,v_1}(x_0),f_{u_2,d,v_2}(x_0)\big)\subseteq A_{x_0}.
\]
Suppose not. Since $f_{u_1,d,v_1}(x)$ and $f_{u_2,d,v_2}(x)$ are boundary points of $A_x$, we get from \eqref{eq:lemgrco3} that
\[
\frac{f_{u_1,d,v_1}(x)+f_{u_2,d,v_2}(x)}{2}\in \Midp(A_x).
\]
Thus there is $u_3 \in \{1,\dots,m\}$ and $v_3 \in \NN$ such that
\[
f_{u_3,d,v_3}(x) = \frac{f_{u_1,d,v_1}(x)+f_{u_2,d,v_2}(x)}{2}.
\]
By \eqref{eq:lemgrco1},
\[
f_{u_1,d,v_1}< f_{u_3,d,v_3} < f_{u_2,d,v_2}.
\]
Hence
\[
\big(f_{u_1,d,v_1}(x_0),f_{u_2,d,v_2}(x_0)\big)\cap A_{x_0} \neq \emptyset.
\]
Thus 
\[
\big(f_{u_1,d,v_1}(x_0),f_{u_2,d,v_2}(x_0)\big)\cap \operatorname{bd}(A_{x_0}) \neq \emptyset.
\]
So there is some $u_4 \in \{1,\dots,m\}$ and $v_4 \in \NN$ such that $\Gamma(f_{u_4,d,v_4}) \subseteq B_1$, and 
\[
f_{u_1,d,v_1}(x_0)< f_{u_4,d,v_4}(x_0) < f_{u_2,d,v_2}(x_0).
\]
By \eqref{eq:lemgrco1},
\[
f_{u_1,d,v_1}(x)< f_{u_4,d,v_4}(x) < f_{u_2,d,v_2}(x).
\]
Since $f_{u_4,d,v_4}(x) \in (B_1)_x$, this contradicts \eqref{eq:lemgrco3}.\\

We now show that $(S_d^*)_{d\in D}$ is a $(w,1)$-grid. We have already seen that $\pi(S_d^*)=\pi(A)$ and thus that $(\pi(S_d^*))_{d\in D}$ is a grid. Let $d,e\in D$ be such that $d>e$. It is left to show that $S_d^*$ is a substack of $S_e^*.$ For this, let $C'$ be a connected component of $S_d^*$.
We know that there are $i,j \in \{1,\dots,m\}$ and $k,\ell \in \NN$ such that
\[
C'=\big(f_{i,d,k},f_{j,d,\ell}\big)
\]
We claim that $C'$ is a connected component of $S_e^*$. We first show that $C'\subseteq S_e^*$.
Let $(x,y)\in C'$.
Then 
\[
\lambda(x,y)=f_{i,d,k}(x) \hbox{ and  } \mu(x,y)=f_{j,d,\ell}(x).
\]
Since $\Gamma(f_{i,d,k})$ is a connected component of $S_{i,d}$, it is also a connected component of $S_{i,e}$. Therefore $f_{i,d,k}=f_{i,e,k'}$ for some $k'\in \NN$, and $f_{i,e,k'}(x)=\lambda(x,y)$. Similarly, we find $\ell'\in \NN$ such that $f_{j,d,\ell}=f_{j,e,\ell'}$ and $f_{j,e,\ell'}(x)=\mu(x,y)$. Thus $(x,y)\in S_e^*$.\\

Towards a contradiction, suppose that $C'$ is not a connected component of $S_e^*$.
Then there are $i',j'\in \{1,\dots,m\}$ and $k',\ell'\in \NN$ such that 
\[
\big(f_{i,d,k},f_{j,d,\ell}\big) = C'\subsetneq \big(f_{i',e,k'},f_{j',e,\ell'}\big).
\]
Thus either $f_{i',e,k'}<f_{i,d,k}$ or $f_{j',e,\ell'}>f_{j,d,\ell}$. Then either
\[
\Gamma(f_{i,d,k}) \subseteq \big(f_{i',e,k'},f_{j',e,\ell'}\big) \hbox{ or } \Gamma(f_{j,d,\ell}) \subseteq \big(f_{i',e,k'},f_{j',e,\ell'}\big).
\]
However, this contradicts that
\[
\Gamma(f_{i,d,k}) \cap A = \emptyset \hbox{ and } \Gamma(f_{j,d,\ell}) \cap A = \emptyset.
\]
Thus $C'$ is a connected component of $S_e^*$.
\end{proof}

\begin{lem}\label{lem:zerogrids} Let $A\subseteq \RR^{n+1}$ be such that $\pi(A)$ is the surface of a $(0,\dots,0)$-grid in $\RR^{n}$, and $A_x$ is discrete for every $x\in \pi(A)$. Then $A$ is the surface of a $(0,\dots,0)$-grid in $\RR^{n+1}$.
\end{lem}
\begin{proof}
Let $\es=(S_d)_{d\in D}$ be a $(0,\dots,0)$-grid in $\RR^n$ such that $\pi(A)=\surf(\es)$.
For every $d\in D$, set 
\[ 
A(d):=\fisol(A,d) \cap \big(\RR^n \times (-\frac1d,\frac1d)\big).
\]
Since the fiber $A_a$ is discrete for all $a\in \pi(A)$ and $D \cup \{0\}$ is closed, we have that
\[
\bigcup_{d\in D} A(d) = A.
\]
Moreover, the family $(A(d))_{d\in D}$ is decreasing and each $A(d)$ is vertically bounded.
Thus, for all $d\in D$ and $a\in \pi(A)$ the fiber $(A(d))_x$ is finite.
Set 
\[
S_d':= A(d) \cap [S_d \times \RR] \text{ and } D':=\{d\in D: \pi(A(d))\cap S_d\neq\emptyset\}.
\]
We claim that $(S_d')_{d\in D'}$ is a $((0,\ldots,0),D')$-grid.\\

First, we show that $S_d'$ is a stack for all $d\in D'$.
Let $d \in D'$. By construction $\pi(S_d')=S_d$ and hence is a stack. Since $\dim S_d = 0$, every connected component of $S_d$ is a singleton set. Let $x \in S_d$ be such that $\{x\}$ is a connected component of $S_d$. Since $A(d)_x$ is finite, we may write $A(d)=\{y_1,\ldots,y_k\}$ for some $k\in \NN$ and $y_1<\cdots<y_k$. For each $i=1,\dots, k$, let $f_i$ be the function mapping $x$ to $y_i$. Then $f_1,\dots,f_k$ witness that $S_d'$ is a stack.\\

Next, we show that $(S_d')_{d\in D'}$ is a grid.
Observe that $\pi(S_d')_{d\in D'}=(S_d)_{d\in D'}$, the latter being a grid by \cref{lem:changeseqs}.
Let $d,e\in D'$ with $d>e$.
Let $C$ be a connected component of $S_d'$.
Then there are $(x,y) \in C$ such that $x\in S_d$ and $C = \{(x,y)\}.$
Since $S_d$ is a substack of $S_e$, the set $\{x\}$ is a connected component of $S_e$.
Moreover, since $(x,y)\in A(d)$, we know that $(x,y)\in A(e)$ as well. Thus $\{(x,y)\}$ is a connected component of $S_e'$.\\

Lastly, we must show that $\bigcup_{d\in D'} S_d'$ is \'etale.
Let $d\in D'$, and let $(x,y)\in S'_d$.
Since $S_d$ is discrete there's some open box $U\subseteq \RR^n$ such that $S_d \cap U=\{x\}$.
Similarly, since $A_x$ is discrete there is an open interval $I\subseteq \RR^n$ such that $A_x \cap I=\{y\}$.
Then $U\times I$ is an open box in $\RR^{n+1}$ and since $(U \times I) \cap A = \{(x,y)\}$ this box witnesses that $A$ is \'etale at $(x,y)$.
Thus $(S_d')_{d\in D'}$ is a grid and $\surf((S_d')_{d\in D'})=A$.
\end{proof}

\begin{lem}\label{lem:subgrid-CC}
Let $w=(w_1,\dots,w_n)\in \{0,1\}^n$, let $(S_d)_{d\in D}$ be a $w$-grid in $\RR^n$, and let $(S_d')_{d\in D}$ be a definable family of $w$-stacks in $\RR^n$ such that for every $d,e\in D$ with $d>e$
\begin{enumerate}
    \item $S_d'$ is a substack of $S_d$, and
    \item $S_d'$ is a substack of $S_e'$.
\end{enumerate}
Then $(S_d')_{d\in D}$ is a $w$-grid.
\end{lem}

\begin{proof}
    We proceed by induction on $n$. For $n=0$, the statement is immediate. 
    Now suppose that $n\geq 1$ and that the statement holds for $n-1$. We first show that $(\pi(S_d'))_{d\in D}$ is a grid. By the assumptions and \cref{lem:projsubstack}, for every $d,e\in D$ with $d>e$,
    the stack $\pi(S'_d)$ is a substack of $\pi(S_d)$ and $\pi(S'_e)$. 
    By induction, it follows that $(\pi(S_d'))_{d\in D}$ is a grid.\\

    For every $d,e\in D$ with $d>e$, the set $S_d'$ is a substack of $S_e'$ by assumption.
    If $w_n=0$, the surface $\bigcup_{d\in D} S_d$ is \'etale. In this situation the union $\bigcup_{d\in D} S_d'$ is \'etale as well, since $\bigcup_{d\in D} S_d'\subseteq \bigcup_{d\in D} S_d$. 
    Thus  $(S_d')_{d\in D}$ is a grid.
\end{proof}

\begin{lem}\label{lem:midgrid}
Let $w \in \{0,1\}^{n}$, and let $\es = (S_d)_{d \in D}\subseteq \RR^{n+1}$ be a $(w,1)$-grid. Then
\[
\Mid(\es) := (\Mid(S_d))_{d \in D}.
\]
is a $(w,0)$-grid over $\pi(\es)$, and for $x \in \surf(\pi(\es))$
\[
\surf(\Mid(\es))_x = \Mid(\surf(\es)_x).
\]
\end{lem}

\begin{proof}
By \cref{lem:stackmid} we know that $\Mid(S_d)$ is a $(w,0)$-stack with $\pi(\Mid(S_d))=\pi(S_d).$ Thus $\pi(\Mid(\es))=\pi(\es)$, and hence is a $w$-grid. For $d,e \in D$ with $d>e$, we have that $S_d$ is a substack of $S_e$, and hence $\Mid(S_d)$ is a substack of $\Mid(S_e)$ by Lemma \ref{lem:midsubstack}.\\

It is left to show that $\bigcup_{d\in D} \Mid(S_d)$ is \'etale. Let $(x,y)\in \Mid(S_d)$ for some $d\in D$.
    Let $C$ be the connected component of $\pi(S_d)$ containing $x$, and let $f,g\in C_{\infty}(C)$ be functions such that $(f,g)$ is the connected component of $S_d$ containing $(x,y)$.
    Then $\Gamma(\Mid(f,g))$ is the connected component of $S_d$ containing $(x,y)$.
    Since $(f,g)$ is open in $\surf(\es)$, there exists some open box $U$ around $(x,y)$ such that 
    \[
    U\cap \surf(\es)\subseteq (f,g).
    \]
    Therefore
    \[ 
    \bigcup_{d\in D} \Mid(S_d) \cap U = \Mid(S_d)\cap U\subseteq \Gamma(\Mid(f,g)).  
    \]
    Hence $\pi$ maps $(\bigcup_{d\in D} S_d) \cap U$ homeorphically onto its image.
    Thus $\bigcup_{d\in D} \Mid(S_d)$ is \'etale, and hence $(\Mid(S_d))_{d\in D}$ is a $(w,0)$-grid.
\end{proof}

The following lemma is crucial when extending the uniformity of the Cantor-Bendixson rank over a connected component of a grid to uniformity over the whole grid. In the proof of the cell decomposition theorem in the o-minimal case, there is no need for such a construction, essentially because the finite union of finite sets is still finite.  However, the countable union of sets with finite Cantor-Bendixson rank does not necessarily have finite Cantor-Bendixson rank.

\begin{lem}\label{lem:UCB}
Let $n\in \NN_{\geq 2}$, let $w  \in \{0,1\}^{n-2}$, and let $\es=(S_d)_{d\in D}$ be a $(w,1,0)$-grid in $\RR^n$ with bounded surface. Then
\[
\{(x,y,z) \in (\RR^{n-2} \times \RR \times \RR) \cap \surf(\es) \ : \ (x,y) \in \surf(\Mid(\pi(\es))) \}
\]
is the surface of a $(w,0,0)$-grid.
\end{lem}

\begin{proof}
    For each $d\in D$, set
    \[
    S_d' := S_d \cap (\Mid(\pi(S_d)) \times \RR).
    \]
    Let $d\in D$. We claim that $S_d'$ is a $(w,0,0)$-stack.
    First observe that 
    \begin{equation}\label{eq:lemucb}
    \pi(S_d')=\pi(S_d) \cap \Mid(\pi(S_d))=\Mid(\pi(S_d)).
  \end{equation}
  
    Thus $\pi(S_d')$ is a $(w,0)$-stack.
    Let $C$ be a connected component of $\Mid(\pi(S_d))$. 
    Then $C$ is contained in a connected component $C'$ of $\pi(S_d)$.
    Thus there are $f_1,\ldots,f_m \in \mathcal{C}(C')$ such that
    \[ 
    S_d \cap [C' \times \RR] = \bigcup_{i=1}^m \Gamma(f_i).
    \]
    It follows that
    \[ S_d' \cap [C \times \RR] = \bigcup_{i=1}^m \Gamma(f_i|_C).\]
    Therefore the functions $f_1|_C, \ldots,f_m|_C$ witness that $S_d'$ is a $(w,0,0)$-stack.\\

    We claim that $(S_d')_{d\in D}$ is a $(w,0,0)$-grid.
    By \cref{lem:midgrid} it follows from \eqref{eq:lemucb} that $(\pi(S_d)')_{d \in D}$ is a $(w,0)$-grid. Let $d,e\in D$ be such that $d>e$. We now show that $S_d'$ is a substack of $S_e'$. Let $C_0$ be a connected component of $S_d'$. 
    Then there is a connected component $C$ of $\Mid(\pi(S_d))$, a connected component $C'$ of $\pi(S_d)$, and $f \in \mathcal{C}(C')$ such that $\Gamma(f)$ is a connected component of $S_d$ and $\Gamma(f|_C)=C_0.$  
    Then $\Gamma(f)$ is also a connected component of $S_e$, the set $C'$ is a connected component of $\pi(S_e)$, and $C$ is a connected component of $\Mid(\pi(S_e))$.
    Thus $\Gamma(f|_C)$ must also be a connected component of $S_e'$, as its connected components are precisely of this form.\\

    It is left to show that $\bigcup_{d\in D} S_d'$ is étale.
    Observe that
    \[ 
    \bigcup_{d\in D} S_d' = \bigcup_{d\in D} \Big(S_d \cap (\Mid(\pi(S_d)) \times \RR)\Big) \subseteq \bigcup_{d\in D} S_d.
    \]
    Since $\bigcup_{d\in D} S_d$ is \'etale, we conclude that $\bigcup_{d\in D} S_d'$ is \'etale as well by \cref{fact:etale}.
\end{proof}

\section{Uniformity of the Cantor-Bendixson rank for grid over cells} \label{section:Uniform-CB-Rank}

We prove the uniformity in \cref{thm:uniform2} in two steps, the first of which is made in this section. It does not depend on d-minimality, or any tameness condition on $\R$. 

\begin{thm}
\label{prop:boundfortowers}
 Let $w\in \{0,1\}^{n-1}$, $\mathcal{S}=(S_d)_{d\in D}$ be a $(w,0)$-grid in $\RR^n$ such that $\surf(\es)$ is vertically bounded, and let $C\subseteq \RR^{n-1}$ be a $w$-cell such that $\pi(S_d)=C$ for all $d\in D$. 
    Then there is $N\in \NN$ such that for all $x\in C$
    \[
    \rk(\cl(\surf(\es)_x))= N.
    \]
\end{thm}

We will later show that we can replace $\surf(\es)$ by any set $A \subseteq \RR^n$ definable in a d-minimal structure, as long as $A_a$ has empty interior for each $a\in \pi(A)$. For that, we have to show \cref{prop:boundfortowers} for arbitrary  $(w,0)$-grids and establish that every set definable in a d-minimal structure is a finite union of surfaces of grids.

\subsection{Proof of \cref{prop:boundfortowers}}
Let $\mathcal{S}=(S_d)_{d\in D}$ be a $(w,0)$-grid, and let $C\subseteq \RR^{n-1}$ be $w$-cell such that $\pi(S_d)=C$ for $d\in D$. Let $g: D \to \NN$ be the function that maps $d$ to the cardinality of a fiber $S_d$. For each $d\in D$, let $f_{S_d,1},\dots,f_{S_d,g(d)}$ be the definable functions witnessing that $S_d$ is a stack over $C$.
Thus
\[
\surf(\es) = \bigcup_{d\in D} \bigcup_{i=1}^{g(d)} \Gamma(f_{S_d,i}).
\]
For ease of notation, we write $f_{d,i}$ instead of $f_{S_d,i}$ for each $d\in D$ and $i\in\{1,\dots, g(d)\}$. Furthermore for each $d \in D$, let $f_{d,0}$ be the function that is constantly $-\infty$, and let $f_{d,g(d)+1}$ be the function that is constantly $+\infty$. Set
\[
B := \{ (u_1,u_2) \in C \times \RR \ : \ u_2 \in \cl(\surf(\es)_{u_1}) \}.
\]

\begin{defn}
   Let $\mu : D \times B \to \NN$ and $\nu : D \times B \to \NN$ be such that 
for every $d\in D$ and $u=(u_1,u_2)\in B$
\begin{enumerate}
    \item $\mu(d,u)$ is the maximal $i\in \{0,1,\dots,g(d)\}$ such that $$f_{d,\mu(d,u)}(u_1) \leq u_2,$$
    \item $\nu(d,u)$ is the the minimal $i\in \{1,\dots,g(d),g(d)+1\}$ such that $$f_{d,\nu(d,u)}(u_1) \geq u_2.$$
\end{enumerate}
\end{defn}

\begin{lem} \label{lem:factsmunu}
Let $u=(u_1,u_2)\in B$ and let $e_1,e_2 \in D$ be such that $e_2<e_1$. Then
\[
f_{e_1,\mu(e_1,u)}(u_1) \leq f_{e_2,\mu(e_2,u)}(u_1) \leq  u_2 \leq f_{e_2,\nu(e_2,u)}(u_1) \leq f_{e_1,\nu(e_1,u)}(u_1),
\]
and hence
\[
\lim_{d \to 0} f_{d,\mu(d,u)}(u_1) \leq  u_2 \leq \lim_{d \to 0} f_{d,\nu(d,u)}(u_1).
\]
\end{lem}
\begin{proof}
The statement follows immediately from the definitions of the functions $\mu$ and $\nu$.   
\end{proof}

\begin{lem}\label{lem:emptybetween}
Let $u=(u_1,u_2)\in B$. Then for all $w \in C$    
\[
\Big(\lim_{d \to 0} f_{d,\mu(d,u)}(w),\lim_{d \to 0} f_{d,\nu(d,u)}(w)\Big)\cap B_{w} = \emptyset.
\]
\end{lem}
\begin{proof}
Towards a contradiction, suppose that there are $v_1 \in C$ and $v_2\in B_{v_1}$ such that
\[
\lim_{d \to 0} f_{d,\mu(d,u)}(v_1)< v_2 <\lim_{d \to 0} f_{d,\nu(d,u)}(v_1).
\]
Set $v:=(v_1,v_2)$. Since $v_2 \in B_{v_1}$, there is $e \in D$ such that one of the following two statements holds:
\begin{enumerate}
    \item $\lim_{d \to 0} f_{d,\mu(d,u)}(v_1) < f_{e,\mu(e,v)}(v_1) \leq v_2$, or
    \item $ v_2 \leq f_{e,\nu(e,v)}(v_1) <\lim_{d \to 0} f_{d,\nu(d,u)}(v_1).$
\end{enumerate}
In both cases, we can deduce that for all $d\in D$
\[
f_{d,\mu(d,u)}(v_1) < f_{e,\mu(e,v)}(v_1)<  f_{d,\nu(d,u)}(v_1).
\]
Thus for all $d \in D$
\[
f_{d,\mu(d,u)}(u_1) < f_{e,\mu(e,v)}(u_1)< f_{d,\nu(d,u)}(u_1).
\]
In particular,
\[
f_{e,\mu(e,u)}(u_1) < f_{e,\mu(e,v)}(u_1)< f_{e,\nu(e,u)}(u_1).
\]
By the maximality of $\mu(e,u)$, we obtain that
\[
f_{e,\mu(e,u)}(u_1) \leq u_2 < f_{e,\mu(e,v)}(u_1) < f_{e,\nu(e,u)}(u_1).
\]
This is a contradiction against the minimality of $\nu(e,u)$.
\end{proof}

\begin{defn}
Let $u,v\in B$. We say $u$ and $v$ have the \textbf{same approximation type} (written $u\sim_{\es} v$) if for every $d\in D$
\[
\mu(d,u)=\mu(d,v) \text{ and } \nu(d,u)=\nu(d,v).
\]
Let $u_1 \in C$. Set 
\[
\sim_{\es,u_1} := \{ (u_2,u_3) \in B_u^2 \ : \ (u_1,u_2) \sim_{\es} (u_1,u_3)\}.
\]
\end{defn}

Note that $\sim_{\es,u_1}$ is an equivalence relation on $B_{u_1}$.

\begin{lem}\label{lem:emptyequiv}
Let $u_1 \in C$ and $u_2,u_3 \in B_{u_1}$ be such that $u_2<u_3$ and $\big(u_2,u_3\big)\cap B_{u_1}=\emptyset$. Then
$u_2 \sim_{\es,u_1} u_3$.
\end{lem}
\begin{proof}
This is again immediate from the definitions of $\mu$ and $\nu$. 
\end{proof}
\noindent As a direct consequence of \cref{lem:emptybetween} and \cref{lem:emptyequiv}, we obtain the following corollary.
\begin{cor}
 \label{lem:lmuequiv}
Let $(u_1,u_2) \in B$. Then
\[
\lim_{d \to 0} f_{d,\mu(d,u)}(u_1) \sim_{\es,u_1}  u_2 \sim_{\es,u_1} \lim_{d \to 0} f_{d,\nu(d,u)}(u_1).
\]
\end{cor}

\begin{lem}\label{lem:equiv2} Let $u_1,v_1 \in C$ and $u_2 \in B_{u_1}$. Then 
\[
\{ v_2 \in B_{v_1} \ : \ (u_1,u_2)\sim_{\es} (v_1,v_2) \} = \{ \lim_{d \to 0} f_{d,\mu(d,u)}(v_1),\lim_{d \to 0} f_{d,\nu(d,u)}(v_1)\}.
\]
\end{lem}
\begin{proof}
Set $u:=(u_1,u_2)$, and for ease of notation, set
\[
u_2^- := \lim_{d \to 0} f_{d,\mu(d,u)}(u_1) \text{ and } u_2^+ := \lim_{d \to 0} f_{d,\nu(d,u)}(u_1)
\]
Similarly, set
\[
v_2^- := \lim_{d \to 0} f_{d,\mu(d,u)}(v_1) \text{ and }  v_2^+ := \lim_{d \to 0} f_{d,\nu(d,u)}(v_1).
\]
Set $v^-:=(v_1,v_2^-)$ and $v^+:=(v_1,v_2^+)$.
We now show that 
\[
v^{-} \sim_{\es} u  \sim_{\es} v^+.
\]
By \cref{lem:emptybetween},
\[
\big(v_2^-,v_2^+\big) \cap B_{v_1} = \emptyset.
\]
Thus by \cref{lem:emptyequiv}, $v_2^{-} \sim_{\es,v_1} v_2^+$. Hence it is enough to show that for all $d\in D$
\[
\mu(d,u)=\mu(d,v^-), \text{ and } \nu(d,u)=\nu(d,v^+).
\]
\noindent We first show that  $\mu(d,u)=\mu(d,v^-)$ for all $d\in D$. Towards a contradiction, suppose that there is $e\in D$ such that $\mu(e,u)\neq \mu(e,v^-)$. Then
\[
f_{e,\mu(e,u)}(v_1) < f_{e,\mu(e,v^-)}(v_1) \leq v_2^-,
\]
yet,
\[
f_{e,\mu(e,u)}(u_1) \leq u_2 <  f_{e,\mu(e,v^-)}(u_1).
\]
Thus for all $d\in D$
\[
f_{d,\mu(d,u)}(u_1) <  f_{e,\mu(e,v^-)}(u_1),
\]
and hence
\[
f_{d,\mu(d,u)}(v_1) <  f_{e,\mu(e,v^-)}(v_1).
\]
Thus $v_2^- = f_{e,\mu(e,v^-)}(v_1)$, contradicting $C\subseteq \fisol(A)$.\newline

Similarly, one can show that  $\nu(d,u)=\nu(d,v^+)$ for all $d\in D$. 
\end{proof}

\begin{lem}\label{lem:typeimplies}
Let $\lambda$ be an ordinal, and let $u=(u_1,u_2)\in B$ and $v_1\in C$ be such that $u_2 \in B_{u_1}^{[\lambda]}$. Then there is $v_2 \in B_{v_1}^{[\lambda]}$ such that $(u_1,u_2) \sim_{\es} (v_1,v_2)$.
\end{lem}

\begin{proof}
We proceed by induction on $\lambda$. For $\lambda = 0$, this is immediate. Now suppose that the claim holds for $\lambda$, and suppose that $u_2 \in B_{u_1}^{[\lambda+1]}$. Thus there is a sequence $(s_i)_{i\in \NN}$ of elements in $B_{u_1}^{[\lambda]}$ such that $\lim_{i\to \infty} s_{i} = u_2$. Reducing to a subsequence, we can assume
that either 
\begin{itemize}
    \item $(s_i)_{i\in \NN}$ is strictly increasing and $s_i < u_2$ for all $i\in \NN$, or
    \item $(s_i)_{i\in \NN}$ is strictly decreasing and $s_i > u_2$ for all $i\in \NN$.
\end{itemize}
 Both cases can be handled similarly. Let us assume the first case holds; that is $s_i< u_2$ for all $i\in \NN$.
By the induction hypothesis, there is a sequence $(t_i)_{i \in \NN}$ of elements of $B_{v_1}^{[\lambda]}$ such that $(u_1,s_i) \sim_{\es} (v_1,t_i)$ for each $i\in \NN$.\newline

Set $v_2:=\lim_{d \to 0} f_{d,\mu(d,u)}(v_1)$ and $v:=(v_1,v_2)$. By Lemma \ref{lem:equiv2} we directly get $u\sim_{\es} v$. Thus it is enough to show that $\lim_{i \to \infty} t_i = v_2.$\newline

We first show that $t_i < v_2$ for all $i \in \NN$. Let $i\in \NN$. Since $s_i < u_2$, there is $e \in D$ such that 
\[
s_i < f_{e,\mu(e,u)}(u_1) \leq u_2.
\]
Thus 
\[
f_{e,\mu(e,(u_1,s_i))}(u_1) < f_{e,\mu(e,u)}(u_1).
\]
Since $u\sim_{\es} v$ and $(u_1,s_i)\sim_{\es} (v_1,t_i)$, then
\[
f_{e,\mu(e,(v_1,t_i))}(v_1) < f_{e,\mu(e,v)}(v_1).
\]
Thus $t_i<v_2$.\newline

Let $\varepsilon >0$. We need to show that there is $j\in \NN$ such that $|v_2 - t_i|< \varepsilon$ for all $i\in \NN_{\geq j}$. Let $d \in D$ be such that $d < \varepsilon$. 
Since $v_2$ is the supremum of a subset of $ B_{v_1}$, there is $e\in D$ such that
\[
v_2 - f_{e,\mu(e,v)}(v_1) < d.
\]
Since $\lim_{i \to \infty} s_i=u_2$, there is $j\in \NN$ such that for all $i\in \NN_{\geq j}$
\[
u_2 > s_i \geq f_{e,\mu(e,u)}(u_1).
\]
Thus for all $i\in \NN_{\geq j}$
\[
\mu(e,u)=\mu(e,(u_1,s_i)).
\]
Since $u\sim_{\es} v$ and $(u_1,s_i)\sim_{\es} (v_1,t_i)$, we get that for all $i\in \NN_{\geq j}$ 
\[
\mu(e,v)=\mu(e,(v_1,t_i)).
\]
Thus
\[
t_i \geq f_{e,\mu(e,v)}(v_1).
\]
Hence $|v_2 - t_i|<\varepsilon$.
\end{proof}

\begin{proof}[Proof of \cref{prop:boundfortowers}]
  Let $u_1,v_1\in C$. Towards a contradiction, suppose that $\rank(B_{u_1})<\rank(B_{v_1})$. Let $N := \rank(B_{u_1})$. Let $\lambda\geq N$. Then $B_{u_1}^{[\lambda]}=\emptyset$. Since $\rank(B_{v_1})>N$, we have $B_{v_1}^{[\lambda]}\neq \emptyset$. Let $v_2 \in B_{v_1}^{[\lambda]}$. By \cref{lem:typeimplies} there is $u_2\in B_{u_1}^{[\lambda]}$.  Contradiction.
\end{proof}

\section{Main theorem}\label{section:Main-Theorem}

In this section, we present the proof of our main results. Finally, we assume that $\R$ is d-minimal.

\begin{defn}
Let $\mathcal{G}$ be a finite collection of grids in $\RR^n$. We say $\mathcal{G}$ is a \textbf{grid decomposition} of $\RR^n$ if the surfaces of elements of $\mathcal{G}$ form a partition of $\RR^n$ and $\{ \pi(\mathcal{S}) \ : \ \es \in \mathcal{G}\}$ is a grid decomposition of $\RR^{n-1}$. The only grid decomposition of $\RR^0$ is $\{\RR^0\}$.
\end{defn}

By \cref{lem:changeseqs}, we can choose the grids  in a grid decomposition $\mathcal{G}$ such that there is a sequence set $D$ such that every $\es \in \mathcal{G}$ is a $(w,D)$-grid for some $w\in \{0,1\}^n$.\newline

We say a set $A\subseteq \RR^n$ is \textbf{compatible} with a grid decomposition $\mathcal{G}$ if for all $\es \in \mathcal{G}$ either $\surf(\mathcal{S})\subseteq A$  or  $A\cap \surf(\mathcal{S}) = \emptyset$.
If $\mathcal{A}$ is a finite collection of subsets of $\RR^n$, we say $\mathcal{A}$ is compatible with the grid decompositon $\mathcal{G}$ if every $A \in \mathcal{A}$ is compatible with $\mathcal{G}$.\newline

As in case of the classical cell decomposition, it follows from compatibility that
\[
A = \bigcup_{\es \in \mathcal{I}} \surf(\es),
\]
where $\mathcal{I} := \{ \es \in \mathcal{G} \ : \ \surf(\es) \subseteq A\}$. Thus, if a set is compatible with a grid decomposition, it is a finite union of surfaces of grids.

\begin{thm}[Grid decomposition]\label{thm:mainthm} Let $n\in \NN$.
\begin{enumerate}
    \item[(GD$)_{n}$] 
    Let $\mathcal{A}$ be a finite collection of definable subsets of $\RR^n$. Then there is a grid decomposition of $\RR^n$ compatible with $\mathcal{A}$.
    \item[(UCB$)_{n+1}$] Let $A\subseteq \RR^{n+1}$ be definable such that $A_a$ has no interior for all $a \in \pi(A)$. 
    Then there is $k\in \NN$ such that for all $a\in \pi(A)$
    \[\rk(\cl(A_a))\leq k.
    \]
    \item[(PC$)_{n+1}$] Let $A\subseteq \RR^{n+1}$ be definable such that $A_a$ is discrete for all $a\in \pi(A)$. 
    Then $A$ is a finite disjoint union of surfaces of grids.
 \end{enumerate}
\end{thm}

Before we begin the proof of \cref{thm:mainthm}, we collect a few immediate corollaries. Of course, the most noteworthy is \cref{thm:uniform2}.

\begin{proof}[Proof of \cref{thm:uniform2}]
Let $A\subseteq \RR^{n+1}$ be definable. Set
\[
A' := \{ (x,y) \in \RR^n \times \RR \ : \ y \in A_x \setminus \inter(A_x)\}.
\]
Note that $A'_x$ has empty interior for all $x \in \RR^n$. Thus, by \cref{thm:mainthm} there is $k\in \NN$ such that for all $x\in \pi(A)$
\[
\rank (\cl(A'_x))\leq k.
\]
Let $x \in \RR^n$. Suppose that $A_x$ has empty interior. By d-minimality, $A_x$ is a finite union of discrete sets, and hence countable. It is enough to show that $A_x$ is the union of $k$ discrete sets. Since $\inter(A_x)=\emptyset$, we have that $A_x = A_x'$. Thus, $\rank(\cl(A_x))\leq k$, and $\cl(A_x)$ is the union of $k$ discrete sets  by \cref{fact:RK}. Hence, since a subset of a discrete set is discrete, we conclude that $A_x$ is the union of $k$ discrete sets.
\end{proof}

\begin{cor}[Piecewise continuity]
  Let  $A\subseteq \RR^n$ and let $f: A \to \RR$ be definable. Then there is a grid decomposition $\mathcal{G}$ of $\RR^n$ compatible with $A$ such that for every $\es\in \mathcal{G}$ with $\surf(\es)\subseteq A$, the restriction $f|_{\surf(\es)}$ is continuous. 
\end{cor}

\begin{proof}
Let $\mathcal{G}_0$ be a grid decomposition of $\RR^{n+1}$ compatible with the graph $\Gamma(f)$. Set
\[
\mathcal{G} := \{ \pi(\es_0) \ : \ \es_0 \in \mathcal{G}_0\}. 
\]
Let $\es \in \mathcal{G}$, and let $\es_0=(S_d)_{d \in D} \in \mathcal{G}_0$ be such that $\pi(\es_0)=\es$. We now show that $f|_{\surf(\es)}$ is continuous. Since all connected components of $\surf(\es)$ are clopen in $\surf(\es)$, it is enough to show that the restriction to a connected component is continuous.\newline

Let $C\subseteq \surf(\es)$ be a connected component of $\surf(\es)$, and let $d \in D$ be such that $C$ is a connected component of $\pi(S_d)$. 
Let $f_1,\dots,f_m \in \mathcal{C}(C)$ witness that $S_d$ is a stack over $\pi(S_d)$. Since $\surf(\es_0)\subseteq \Gamma(f)$, we know that $m=1$. Thus
\[
\Gamma(f_1) = S_d \cap [C \times \RR] \subseteq \Gamma(f).
\]
Thus $f|_C=f_1$, and hence continuous.
\end{proof}

\begin{cor}
  Let $A \subseteq \RR^n$ be definable. Then
  \begin{enumerate}
      \item $A$ is constructible, and hence $\DSig$, and
      \item $\dim A = \max \{ w_1+\dots+w_n \ : \ A \text{ contains the surface of a $(w_1,\dots w_n)$-grid}\}$.
  \end{enumerate}
\end{cor}
\begin{proof}
For (1), observe that $A$ is a finite disjoint union of surfaces of grids by \cref{thm:mainthm}, and thus a finite disjoint union of locally closed sets by \cref{lem:proj-grid}. Hence $A$ is constructible.\newline

For (2), first observe that by \cref{lem:proj-grid} and \cref{fact:naive}(2)
\[
\dim A \geq \max \{ w_1+\dots+w_n \ : \ A \text{ contains the surface of a $(w_1,\dots w_n)$-grid}\}.
\]
 By \cref{thm:mainthm}, there are grids $\es_1,\dots,\es_m$ in $\RR^n$ such that $A = \bigcup_{i=1}^m \surf(\es_i)$. By \cref{fact:naive}(4),
\[
\dim A = \max \{\dim(\surf(\es_1)),\dots,\dim(\surf(\es_m))\}.
\]
Thus (2) follows from \cref{lem:proj-grid}.
\end{proof}

\begin{cor}
 Let $A\subseteq \RR^n$ be definable. Then there is a decreasing definable family $(A_d)_{d \in D}$ such that $A = \bigcup_{d \in D} A_d$ and for every $d \in D$
 \begin{enumerate}
     \item $A_d$ has finitely many connected components, and
     \item $\dim A = \dim A_d$.
 \end{enumerate}
\end{cor}

\begin{proof}
Let $\es_1=(S_{1,d})_{d\in D},\dots,\es_m=(S_{m,d})_{d\in D}$ be grids in $\RR^n$ such that 
\[
A = \bigcup_{i=1}^m \surf(\es_i).
\]
For $d \in D$, set
\[
A_d := \bigcup_{i=1}^m S_{i,d}.
\]
Since stacks have finitely many connected components, so does $A_d$. We observe that 
\begin{align*}    
\dim A &= \max \{ \dim \surf(\es_i) \ : \ i=1,\dots, m\}\\
&= \max \{ \dim S_d \ : \ i=1,\dots, m\}\\
&= \dim A_d. 
\end{align*}
Since grids are definable families, so is $(A_d)_{d\in D}$.
\end{proof}

We recall the definition of a \textbf{countable cell decomposition}\footnote{This is called countable decomposition in \cite{Miller-tame}. We added the word ``cell'' in order to distinguish it from the grid decomposition results.} from \cite{Miller-tame}: the only countable cell decomposition of $\RR^0$ is $\{\RR^0\}$, and for $n>0$, a countable cell decomposition of $\RR^{n}$ is a countable partition $\mathcal{D}$ of $\RR^n$ into cells such that $\{ \pi(C) \ : \ C \in \mathcal{D}\}$ is a countable cell decomposition of $\RR^{n-1}$.

\begin{lem}\label{lem:ccd}
Let $\mathcal{G}$ be a grid decomposition of $\RR^n$. Then
\[
\mathcal{D} := \{ C\subseteq \RR^n \ : \ \es \in \mathcal{G}, \ C \text{ is a connected component of } \surf(\es)\}
\]
is a countable cell decomposition of $\RR^n$.
\end{lem}
\begin{proof}
We proceed by induction on $n$. The case $n=0$ is immediate. Let $n>0$. It is clear that $\mathcal{D}$ is partition. Thus it is left to show that $\{ \pi(C) \ : \ C \in \mathcal{D}\}$ is a countable cell decomposition of $\RR^{n-1}$. By definition, $\{ \pi(\es) \ : \ \es \in \mathcal{G}\}$ is a grid decomposition of $\RR^{n-1}$. By induction,
\[
\mathcal{D}':=\{ C'\subseteq \RR^{n-1} \ : \ \es \in \mathcal{G}, \ C' \text{ is a connected component of } \surf(\pi(\es))\}
\]
is a countable cell decomposition. It remains to argue that for $C'\in \mathcal{D}'$ there is $C\in \mathcal{D}$ with $\pi(C)=C'$. Let $C'\in \mathcal{D}'$ and let $\es=(S_d)_{d\in D} \in \mathcal{G}$ be such that $C'$ is a connected component of $\surf(\pi(\es))$. By \cref{lem:CC0-grid}, there is $d \in D$ such that $C'$ is a connected component of $\pi(S_d)$. By \cref{lem:fin-CC} there is a connected component $C$ of $S_d$ such that $\pi(C)=C'.$ Since $C$ is a connected component of $\surf(\es)$ by \cref{lem:CC0-grid}, we conclude that $C\in \mathcal{D}$.
\end{proof}

\begin{cor}[Countable cell decomposition]
Let $\mathcal{A}$ be a finite collection of definable subsets of $\RR^n$. Then there is a countable cell decomposition of $\RR^n$ compatible with $\mathcal{A}$.    
\end{cor}
\begin{proof}
Let $\mathcal{G}$ be a grid decomposition compatible with $\mathcal{A}$. Then
\[
\mathcal{D} := \{ C\subseteq \RR^n \ : \ \es \in \mathcal{G}, \ C \text{ is a connected component of } \surf(\es)\}
\]
is compatible with $\mathcal{A}$, and by \cref{lem:ccd} a countable cell decomposition.
\end{proof}

\subsection{Proof of \cref{thm:mainthm}}
In the next three subsections, we show the following three implications:
\begin{itemize}
    \item[\eqref{section:step1}] (PC$)_{n}$+(UCB$)_{n-1}\Rightarrow $ (UCB$)_{n}$, for $n\geq 2$,
    \item[\eqref{section:step2}] (GD$)_{n-1}$+(UCB$)_{n}$+(PC$)_{n}\Rightarrow $ (GD$)_{n}$, for $n\geq 1$,
    \item[\eqref{section:step3}] (GD$)_{n}\Rightarrow $ (PC$)_{n+1}$, for $n\geq 1$.
\end{itemize}
Note that (GD$)_0$ holds trivially, and (PC$)_1$ holds since every discrete set is the surface of a $(0)$-grid by \cref{lem:zerogrids}. Furthermore, (UCB$)_1$ follows immeaditely from d-minimality, and we deduce (GD$)_1$ from \eqref{section:step2}. Thus \cref{thm:mainthm} follows from the three implications and simultaneous induction on $n$. 

\subsubsection{(PC$)_{n}$+(UCB$)_{n-1}\Rightarrow $ (UCB$)_{n}$}\label{section:step1}
Let $n\geq 2$, and let $A\subseteq \RR^n$ be definable such that $A_a$ has no interior for each $a \in \pi(A)$. 
Without loss of generality, we can assume that $\pi(A)\subseteq [0,1]^{n-1}$. By \cref{fact:cb-homeo}, we also may assume that $A_a\subseteq [0,1]$ for each $a \in \pi(A)$.
Set 
\[
C := \{ (a,b) \in \RR^{n-1} \times \RR \ : \ b \in \isol(A_a)\}.
\]
Note that $C\subseteq A$ and $C_a$ is discrete and non-empty for each $a \in \pi(A)$. Moreover, for each $a\in \pi(A)$,
\[
\cl(C_a) = \cl(A_a).
\]
Thus, we can reduce to the case that $A=C$, and hence to the case that $A_a$ is discrete for each $a\in \pi(A)$. 
By (PC$)_{n}$ there are grids $\es_1,\ldots,\es_m$ in $\RR^n$ whose surfaces partition $A$. 
By \cref{fact:cbmax}, we have that for each $a\in \pi(A)$
\begin{align*}
\rk(\cl(A_a)) &= \rk(\cl(\bigcup_{i=1}^m \surf(\es_i)_a)) = \rk( \bigcup_{i=1}^m\cl(\surf(\es_i)_a))\\
&= \max \{ \rk(\cl(\surf(\es_i)_a)) \ : \ i=1,\dots,m\}.
\end{align*}
Hence, we can further reduce to the case that $A$ is the surface of a grid. Let $w=(w_1,\dots,w_n)\in \{0,1\}^n$ be such that $A$ is the surface of a $w$-grid.\\

We first consider the case that $w_{n-1} = 0$. Then the fiber $\pi(A)_x$ is discrete for each $x \in \RR^{n-2}$.
Let $\mu : \pi(A) \to \RR_{\geq 0}$ be given by
\[
(x,y) \in \RR^{n-2} \times \RR \mapsto \frac13 \sup \{ t \in \RR_{\geq 0} \  : \ (y-t,y+t) \cap \pi(A)_x = \{y\}\}. 
\]
Note $\mu$ is definable, since $\pi(A)$ is definable. Since $\pi(A)_x$ is discrete for each $x\in \RR^{n-2}$, we have that $\mu(\pi(A))\subseteq \RR_{>0}$.
Define
\[
B := \bigcup_{(x,y) \in \pi(A)} \{ (x,y + \mu(x,y)z) \ : \ z \in \cl(A_{(x,y)}) \}.
\]
Observe that $B$ is definable. Since $\pi(A)_x\subseteq [0,1]$, we have that $y + \mu(x,y)\cl(A_{(x,y)})$ is closed and open in $B_x$ for each $x\in \RR^{n-2}$ and $y\in \RR$. It follows that for each $x\in \RR^{n-2}$ the fiber $B_x$ does not have interior and
\[
\cl(B_x) = \bigcup_{y \in \pi(A)_x}y + \mu(x,y)\cl(A_{(x,y)}) = B_x.
\]
Thus, by (UCB$)_{n-1}$, there is $k \in \NN$ such that for all $x \in \RR^{n-2}$
\[
\rk(\cl(B_x)) \leq k.
\]
Now let $(x,y)\in \pi(A)$. Since the sets $\cl(A_{(x,y)})$ and $y+\mu(x,y)\cl(A_{(x,y)})$ are homeomorphic, we obtain by \cref{fact:cb-homeo} that 
\[
\rank(\cl(A_{(x,y)})) = \rank(y + \mu(x,y)\cl(A_{(x,y)})) \leq \rank(\cl(B_x))\leq k.
\]
This concludes the case where $w_{n-1}=0.$\newline

Now consider the case that $w_{n-1}=1$. By \cref{lem:UCB}
\[
A' := \{(x,y,z) \in (\RR^{n-2} \times \RR \times \RR)\cap A \ : \ (x,y) \in \surf(\Mid(\pi(A)))\},
\]
is the surface of a $(w_1,\dots,w_{n-2},0,0)$-grid. By the first case, there is $k\in \NN$ such that for all $(x,y)\in \surf(\Mid(\pi(A)))$
\[
\rank(\cl(A'_{(x,y)}))\leq k.
\]
We now show that $k$ has the desired property. Let $(x,y,z)\in \RR^{n-2}\times \RR \times \RR$ be such that $(x,y,z)\in A$ and let $C$ be the connected component of $\pi(A)$ containing $(x,y)$. Observe that the connected component of $\pi(A)_x$ containing $y$ is a bounded open interval. Let $e \in \RR$ be its midpoint, and note that  $(x,e) \in C$. Hence by \cref{prop:boundfortowers}, 
\[
\rank(\cl(A_{(x,y)})) = \rank(\cl(A_{(x,e)})).
\]
Since $A'_{(x,e)}=A_{(x,e)}$, we conclude that 
\[
\rank(\cl(A_{(x,y)})) = \rank(\cl(A_{(x,e)})) = \rank(\cl(A'_{(x,e)}))\leq k.
\]

\subsubsection{ (GD$)_{n-1}$+(UCB$)_{n}$+(PC$)_{n}\Rightarrow $ (GD$)_{n}$}\label{section:step2}

First, we collect two lemmas we need in the proof. The first one is an easy consequence of (UCB$)_{n}$ and (PC$)_{n}$.

\begin{lem}\label{lem:discretefibers}
    Let $A\subseteq \RR^n$ be definable such that $A_x$ has no interior for all $x\in \pi(A)$.
    Then $A$ is a finite disjoint union of surfaces of grids.
\end{lem}

\begin{proof}
    By (UCB)$_{n}$, there is $M\in \NN$ such that  for all $x\in \pi(A)$
    \[
    \rank(\cl(A_x))\leq M.
    \]
    For each $j=1,\ldots,M,$ define a set 
    \[
    X_j := \{(x,y)\in \RR^{n-1}\times \RR: y\in A_x \cap \isol(\cl(A_x)^{[j]})\}.
    \]
    Observe that for all $j,\ell \in \{1,\dots,M\}$
    \begin{enumerate}
        \item $A=\bigcup_{j=1}^{M} X_j$,
        \item $X_j$ is definable,
        \item $X_{j}\cap X_{\ell} = \emptyset$, whenever $j\neq \ell$, and
        \item $(X_j)_x$ is discrete for all $x\in \pi(X_j)$.
    \end{enumerate}
    By (PC)$_{n}$, the set $X_j$ is a finite disjoint union of surfaces of grids for $j\in \{1,\dots,M\}$. 
    Taking the union of all such surfaces, we conclude that $A$ is a finite disjoint union of surfaces of grids.
\end{proof}

The next lemma uses (GD$)_{n-1}$ to refine grid decompositions.
\begin{lem}\label{lem:grid-refine}
Let $\mathcal{A}=\{A_1,\ldots,A_k\}$ be a finite collection of definable subsets of $\RR^n$ such that $\mathcal{A}$ partitions $\RR^n$ and each $A_i\in\mathcal{A}$ is a finite disjoint union of surfaces of grids.
Then there is a grid decomposition of $\RR^n$ compatible with $\mathcal{A}$.
\end{lem}

\begin{proof}
    Let $\mathcal{A}=\{A_1,\dots,A_k\}$ be a collection of definable subsets of $\RR^n$ that partition $\RR^n$.
    For each $i=1,\ldots,k$ let $\es_{i,1},\ldots, \es_{i,p(i)}$ be the grids such that 
    \[
    \bigcup_{j=1}^{p(i)} \surf(\es_{i,j})=A_i.
    \]
 
    By (GD)$_{n-1}$, there is a grid decomposition $\G$ of $\RR^{n-1}$ compatible with
    \[
    \{ \surf(\pi(\es_{i,j}))\ : \ i\in \{1,\ldots,k\}, j\in \{1,\ldots,p(i)\}\}. 
    \]

    Let $i\in \{1,\ldots,k\}$  and  $j\in \{1,\ldots,p(i)\}$. 
    For ease of notation, we write $\es=(S_d)_{d\in D}$ for $\es_{i,j}$. 
    Now let $\es'= (S'_d)_{d\in D}\in \G$ be such that 
    \[
    \surf(\es') \subseteq \surf(\pi(\es)).
    \]
    Let $w=(w_1,\ldots,w_n)\in \{0,1\}^n$ be such that $\es$ is a $w$-grid, and let $w'=(w'_1,\ldots,w'_{n-1})\in \{0,1\}^{n-1}$ be such that $\es'$ is a $w'$-grid.
    Set 
    \[
    D':=\{d\in D: \pi(S_d) \cap S_d' \neq \emptyset \}
    \]
    and for $d\in D'$, set
    \[
    S^*_d:=S_d \cap [S_d' \times \RR].
    \]
    We claim that $(S^*_d)_{d\in D'}$ is a $(w_1',\ldots,w_{n-1}',w_n)$-grid in $\RR^n$.\\

    Let $d\in D$. We first show that $S_d^*$ is $(w_1',\ldots,w_{n-1}',w_n)$-stack. By \cref{lem:fin-CC} and \cref{lem:restrict-stack}, it is enough to prove that $\pi(S^*_d)$ is a substack of $S_d'$. By \cref{lem:substack-CC} it is sufficient to establish that every connected component of $\pi(S^*_d)$ is a connected component of $S_d'$.\\
    
    Let $C\subseteq \RR^{n-1}$ be a connected component of $\pi(S^*_d)=\pi(S_d)\cap S_d'$.
    Then there is a connected component $C'$ of $S_d'$ containing $C$.
    Since $\surf(\es')\subseteq \surf(\pi(\es))$, there is a connected component $C''$ of $\surf(\pi(\es))$ containing $C'$.
    By \cref{lem:CC0-grid} there is $e\in D$ such that $C''$ is a connected component of $\pi(S_e)$. 
    If $d\leq e$, then $C''$ is a connected component of $\pi(S_d)$ as well.
    On the other hand, if $e<d$, then either $C''$ is a connected component of $\pi(S_d)$ or $C''$ is disjoint from $\pi(S_d)$.
    Since $C''$ is not disjoint from $\pi(S_d)$, we must have that $C''\subseteq \pi(S_d)$.
    Therefore $C'\subseteq \pi(S_d)$ as well.
    Thus $C'\subseteq \pi(S_d)\cap S_d'$ and so $C'=C$. Hence $C$ is a connected component of $S'_d$.\\

    Let $d,e\in D'$ with $e<d$. We now show that $S_d^*$ is a substack of $S_e^*$. By \cref{lem:substack2}, it is enough to show that  $\pi(S_d^*)$ is a substack of $\pi(S_e^*)$. Let $C$ be a connected component of $\pi(S_d^*)$. Suppose towards a contradiction that $C$ is not a connected component of $\pi(S_e^*)$.
    Thus $C$ is strictly contained in a connected component $C'$ of $\pi(S_e^*)$.
    By the above, the set $C'$ is also a connected component of $S_e'$. 
    Since $C$ is a connected component of $\pi(S_d^*)$ it is a connected component of $S_d'$, and thus a connected component of $S_e'$. Thus $C=C'$, a contradiction.\\

    By \cref{lem:subgrid-CC} we have that $(\pi(S_d*))_{d\in D'}$ is a $(w_1',\ldots,w_{n-1}')$-grid.
    Lastly, suppose that that $w_n=0$. Since 
    $\bigcup_{d\in D'} S_d^* \subseteq \bigcup_{d\in D} S_d$
    and the latter being \'etale, we conclude  that $\bigcup_{d\in D'} S_d^*$ is \'etale as well.
    Therefore $(S_d^*)_{d\in D'}$ is a grid.\\

    Observe that $\surf((S_d^*)_{d\in D'})=\surf(S)\cap [\surf(S')\times \RR]$, and thus the collection of all such grids has the desired properties.
\end{proof}

We are now ready to prove (GD$)_n$. Let $A_1,\dots, A_m\subseteq \RR^m$ be definable such that $\mathcal{A}=\{A_1,\dots,A_m\}$. Without loss of generality we may assume that $A_i\cap A_j=\emptyset$ whenever $i,j\in\{1,\dots,m\}$ and $i\neq j$.
Set 
\[ A_{m+1}:= \RR^n \setminus \left (\bigcup_{i=1}^m A_i \right).\]
We will show that $A_i$ can be written as a finite disjoint union of surfaces of grids for each $i=1,\ldots,m+1$.\\

Let $i\in \{1,\ldots,m,m+1\}$. For ease of notation, set $A:=A_i$.
Let
\[
X := \{(x,y)\in \RR^{n-1}\times \RR: (x,y)\in A \land y\not\in \inter(A_x)\}.
\]
We can take $X$ to be a finite disjoint union of surfaces of grids by \cref{lem:discretefibers}.\\

Now let
\[
Y := \{(x,y)\in \RR^{n-1}\times \RR :  y\in \inter(A_x)\} .
\]
Then set 
\begin{align*}
Z_1&:=\{(x,y)\in \RR^{n-1}\times \RR: (x,y)\in A \land y\in \bd(Y_x)\},\\
Z_2&:=\{(x,y)\in \RR^{n-1}\times \RR: (x,y)\in A \land y\in \Midp(Y_x)\}.
\end{align*}
By \cref{lem:discretefibers} both $Z_1$ and $Z_2$ are finite disjoint unions of surfaces of grids.
Applying (GD)$_{n-1}$, we know that $\pi(Y)$ is a finite disjoint union of surfaces of grids.
Therefore by \cref{cor:grid-complement} we find that $Y$ is a finite disjoint union of surfaces of grids.
Since $X$ and $Y$ are disjoint and $X\cup Y=A$, we conclude that $A$ is a finite disjoint union of surfaces of grids.\\

Thus we have shown that $A_i$ can be written as a finite disjoint union of surfaces of grids for $i=1,\dots,m+1$.
Now apply \cref{lem:grid-refine} to obtain a grid decomposition of $\RR^n$ compatible with $A_1,\ldots,A_m$, as desired.

\subsubsection{(GD$)_{n}$+(PC)$_{n} \Rightarrow $ (PC$)_{n+1}$}\label{section:step3}

Let $A\subseteq \RR^{n+1}$ be definable such that $A_x$ is discrete for all $x\in \pi(A)$.
We will prove (PC$)_{n+1}$ by induction on the dimension of $\pi(A)$.\\

Suppose that $\dim(\pi(A))=0$.
By (GD)$_{n}$, we can write $\pi(A)$ as a finite disjoint union of surfaces of grids. Thus, we can reduce to the case that $\pi(A)$ is the surface of a grid.
Let $w=(w_1,\ldots,w_n)\in \{0,1\}^n$ and let $\es=(S_d)_{d\in D}$ be a $w$-grid in $\RR^n$ such that $\pi(A)=\surf(\es)$.
By \cref{cor:grid-dim}, 
\[
w_1=\cdots=w_n=0.
\]

Since the fiber $A_x$ is discrete for all $x\in \pi(A)$,
we can conclude that $A$ is the surface of a $(0,\dots,0)$-grid by  \cref{lem:zerogrids}.\\

Inductively assume that (PC)$_{n+1}$ holds for every $k<\dim(\pi(A))$.
We will need the following lemma about \'etale points.

\begin{lem}\label{cor:etale-grids}
    Let $A\subseteq \RR^{n+1}$ be definable and vertically bounded such that $\pi(A)$ is the surface of a grid and $A_x$ is finite for all $x\in \pi(A)$.
    Then 
  \[
  \{x\in \pi(A): \text{ there is $y\in \RR$ such that } A \text{ is not \'etale at } (x,y)\}   
  \]
is nowhere dense in $\pi(A)$.
\end{lem}

    \begin{proof}
        Let 
        \[
        Y:= \{(x,y)\in A: A \text{ is not \'etale at } (x,y)\},
        \]
        and note that $Y$ is definable. 
        Suppose towards a contradiction that $\pi(Y)$ is somewhere dense in $\pi(A)$. 
        Then $\pi(Y)$ has interior in $\pi(A)$ by \cref{cor:boundary}(1).
        Applying (GD)$_{n}$, we obtain a grid decomposition of $\RR^n$ compatible with $\pi(A)$ and $\pi(Y)$.
        Therefore, we may write $\pi(Y)$ as a finite union of surfaces of grids.
        Since $\pi(Y)$ has interior in $\pi(A)$, by Lemma \ref{lem:grid-inter} there is a grid $\es$ in $\RR^n$ whose surface is contained in $\pi(Y)$ and has interior in $\pi(A)$. Let $f : \surf(\es) \to \RR$ be the definable function which maps $x \in \surf(\es)$ to the minimum of $Y_x$. Since $Y_x$ is finite for all $x \in \pi(A)$, this function $f$ is well-defined.
         By \cref{cor:gridcont}, there is a dense, definable, open set $Z\subseteq \surf(\es)$ such that $f|_Z$ is continuous.
        Applying (GD)$_{n}$ again, we obtain a grid decomposition of $\RR^n$ compatible with $Z$ and $\pi(A)$.
        Since $Z$ has interior in $\pi(A)$, there is a grid $\es'$ by Lemma \ref{lem:grid-inter} such that $\surf(\es')\subseteq Z$ and $\surf(\es')$ has interior in $\pi(A)$.\newline

        \noindent Now define sets $A'_0, A'_a, A'_b$, and $A'_{ab}$ as follows:
        \begin{align*}
            A'_0 &:= \{x \in \surf(\es') \ : \ A_x=\{f(x)\} \},\\
            A'_a &:= \{x\in \surf(\es') : \forall y\in A_x \ y\geq f(x)\}, \\
            A'_b &:= \{x\in \surf(\es'): \forall y\in A_x \  y \leq f(x)\},\\
            A'_{ab} &:= \{x\in \surf(\es') : \exists y_1, y_2 \in A_x \ (y_1<f(x)) \land (y_2>f(x))\}.
        \end{align*}
        Each of these sets is definable and 
        \[
        \surf(\es')=A'_0\cup A'_a \cup A'_b \cup A'_{ab}.
        \]
        Using that $\surf(\es')$ has interior in $\pi(A)$, apply (GD)$_{n}$ and Lemma \ref{lem:grid-inter} in the same way as before to obtain a grid $\es''$ whose surface is contained in one of $A'_0,A'_a, A'_b,$ or $A'_{ab}$ and has interior in $\pi(A)$.
        We will only consider the case when $\surf(\es'')\subseteq S_{ab}$, as the other cases follow similarly.\\

        Let $f_a,f_b : \surf(\es'')\to \RR$ be definable function such that for all $x\in \surf(\es'')$
         \begin{align*}
            f_a(x) &:= \min{\{y':y'>\min{Y_x}\}}),\\
            f_b(x) &:= \max{\{y':y'<\min{Y_x}\}}).
        \end{align*}
        By \cref{cor:gridcont}, there are dense, definable, open sets $U_a, U_b \subseteq \surf(\es'')$ such that the restriction $f_a|_{U_a}$ and $f_b|_{U_b}$ are continuous.
        Set $U:=U_a \cap U_b$ and observe that $U$ is a dense, definable, and open subset of $\surf(\es'')$ on which $f,f_a$ and $f_b$ are continuous.\newline
        
        Let $x\in U$. We claim that $A$ is \'etale at $(x,f(x))$. 
        By construction, the graphs of $\Gamma(f|_U), \Gamma(f_a|_U)$, and $\Gamma(f_b|_U)$ do not intersect.
        Therefore, there is an open box $B\subseteq \RR^{n+1}$ centered at $(x,f(x))$ such that 
        \[
        A\cap B\subseteq \Gamma(f)\cap [U \times \RR].
        \]
        Thus $\pi$ maps $A\cap B$ homeomorphically onto an open set, and $A$ is \'etale at $(x,f(x))$. 
        However, since $(x,f(x))\in Y$, this is a contradiction.
    \end{proof}

Now we continue with our proof of (PC)$_{n+1}$.
Apply (GD)$_{n}$ to obtain a grid decomposition $\mathcal{G}$ of $\RR^n$ compatible with $\pi(A)$.
Without loss of generality, we can assume that $\pi(A)$ is the surface of one of the grids in $\mathcal{G}$.\\

For each $e\in D$, set 
\[ A(e) := \fisol(A,e) \cap \big(\RR^n \times (-\frac1e,\frac1e)\big).\]
Since  $A_x$ is discrete for all $x\in \pi(A)$ and $D \cup \{0\}$ is closed, 
\[
A= \bigcup_{e\in D} A(e).
\]
Moreover, the family $(A(e))_{e \in D}$ is decreasing and each $A(e)$ is vertically bounded.
Thus, for all $e\in D$ and $x\in \pi(A)$, the fiber $A(e)_x$ is finite.\\

By \cref{cor:boundary}(3), the boundary of $\pi(A(e))$ in $\pi(A)$ is nowhere dense. 
Hence, by \cref{lem:dim}, the union $\bigcup_{e\in D} \operatorname{bd}_{\pi(A)}(\pi(A(e))$ is nowhere dense in $\pi(A)$. 
By \cref{cor:gridnowheredense}(2) and induction on $\dim(\pi(A))$, we conclude that
\[
A \cap \Big[\big(\bigcup_{e\in D} \operatorname{bd}_{\pi(A)}(\pi(A(e))\big)\times \RR\Big]
\]
is a finite union of surfaces of grids. 
It is left to show that
\[
A \cap \Bigg[\Big(\pi(A) \setminus \big(\bigcup_{e\in D} \operatorname{bd}_{\pi(A)}(\pi(A(e))\big)\Big)\times \RR\Bigg]
\]
is a finite union of surfaces of grids. 
By (GD$)_{n}$,
\[\pi(A) \setminus \big(\bigcup_{e\in D} \operatorname{bd}_{\pi(A)}(\pi(A(e))\big)\]
is a finite union of surfaces of grids. 
Let $\es=(S_d)_{d \in D}$ be one of these grids. 
It is enough to show that
\[
A \cap (\surf(\es) \times \RR)
\]
is a finite union of surfaces of grids. \\

Let $C\subseteq \RR^n$ be a connected component of $\surf(\es)$. 
We claim that for all $e\in D$
\[
\pi(A(e)) \cap C =\emptyset \text{ or } C\subseteq \pi(A(e)).
\]
Indeed, since $C \cap \operatorname{bd}_{\pi(A)}(\pi(A(e)))=\emptyset$, we know that 
\[
C \subseteq \inter_{\pi(A)}(\pi(A(e))) \cup \inter_{\pi(A)}(\pi(A)\setminus \pi(A(e))).
\]
From connectedness of $C$, we deduce that 
\[
C \subseteq \inter_{\pi(A)}(\pi(A(e))) \text{ or } C \subseteq \inter_{\pi(A)}(\pi(A)\setminus \pi(A(e))).
\]
The claim follows.
\\

Let $e\in D$. Applying Corollary \ref{cor:etale-grids} to $A(e)\cap [\surf(\es) \times \RR]$, we know that
\[ E_e:=\{x\in \surf(\es) : \text{there is } y \in \RR \text{ such that } A(e) \cap [\surf(\es) \times \RR] \text{ is not \'etale at } (x,y) \}\]
is nowhere dense in $\surf(\es)$.
Similarly, applying \cref{cor:FL-grids} to $A(e) \cap [\surf(\es) \times \RR]$, we conclude that 
\[ 
F_e:=\{x\in \surf(\es) : \cl(A(e)_x)\neq \cl(A(e))_x \}
\]
is nowhere dense in $\surf(\es)$.
Thus $E_e\cup F_e$ is nowhere dense in $\surf(\es)$, and hence $\dim (E_e \cup F_e) < \dim(\surf(\es))$ by \cref{cor:gridnowheredense}. By \cref{lem:dim}
\[
\dim \big(\bigcup_{e \in D} (E_e \cup F_e) \big) = \max_{e \in D} (\dim (E_e \cup F_e)) < \dim(\surf(\es)).
\]
Set 
\[
V := \surf(\es) \setminus \big(\bigcup_{e\in D} (E_e \cup F_e)\big).
\]
Note that $V\subseteq \surf(\es)$ is dense, definable, and open in $\surf(\es)$, and for every $e\in D$
\begin{enumerate}
    \item $A(e) \cap (\surf(\es)\times \RR)$ is étale at every $(x,y) \in A(e)\cap (V\times \RR)$,
    \item $\cl{(A(e)_x)}=\cl{(A(e))}_x$ for all $x \in V$.
\end{enumerate}

By induction on $\dim(\pi(A))$, the set 
\[
A \cap \big((\surf(\es)\setminus V)\times \RR\big)
\]
is a finite union of surfaces of grids. It is left to show that $A\cap (V\times \RR)$ is a finite union of surfaces of grids. Observe that $V$ is a finite union of surfaces of grids by (GD$)_{n}$.
Let $\es' = (S'_d)_{d \in D}$ be one of these grids, and let $(w_1,\dots,w_n) \in \{0,1\}^n$ be such that $\es'$ is a $(w_1,\ldots,w_n)$-grid. 
It is enough to show that
\[
A \cap (\surf(\es')\times \RR)
\]
is a finite union of surfaces of grids.\\

For every $d \in D$, we set
\[
S''_{d} := A(d) \cap (S'_{d} \times \RR).
\]
It is left to show that $\es'':=(S''_d)_{d \in D}$ is a $(w_1,\ldots,w_n,0)$-grid, and 
\[
\surf(\es'') = A \cap (\surf(\es')\times \RR).
\]

Let $d\in D$.
We claim that $S_d''$ is a $(w_1,\ldots,w_n,0)$-stack. 
First, we show that $\pi(S_d'')$ is a $(w_1,\ldots,w_n)$-stack.
By \cref{lem:substack-CC}, to show this and that $\pi(S_d'')$ is a substack of $S_d'$, it is enough to prove that every connected component of $\pi(S_d'')$ is a connected component of $S_d'$.
Let $C$ be a connected component of $\pi(S_d'')$, and let $C'$ be a connected component of $S_d'$ containing $C$.
Since $S_d'\subseteq V\subseteq \surf(\es)$, we know that $C'$ is contained in a connected component of $\surf(\es)$. Therefore either 
\[
C'\subseteq \pi(A(d)) \text{ or }C'\cap \pi(A(d))=\emptyset.
\]
Since $C' \cap \pi(A(d))\neq \emptyset$, we have that $C'\subseteq \pi(A(d))$. Thus $C'\subseteq \pi(S_d'')$, and hence $C=C'$.\\

Thus $\pi(S_d'')$ is a $(w_1,\ldots,w_n)$-stack. 
To conclude that $S_d''$ is a $(w_1,\ldots,w_n,0)$-stack, it is enough to verify the assumptions of \cref{lem:4.11}. However, these follow from \cref{fact:etale} and \cref{fact:opensubset}.\\

Let $d,e\in D$ be such that $d>e$. We now show that $S_d''$ is a substack of $S_e''$. By \cref{lem:substack}, it is enough to show that $\pi(S_d'')$ is a substack of $\pi(S_e'')$. Let $C$ be a connected component of $\pi(S_d'')$. Since $\pi(S_d'')$ is a substack of $S_d'$, we know that $C$ is a connected component of $S_d'$. Thus $C$ is a connected component of $S_e'$, because $S_d'$ is a substack of $S_e'$. 
Since 
\[
C\subseteq \pi(A(d))\subseteq \pi(A(e)),
\]
the set $C$ is  a connected component of $\pi(A(e))\cap S_e'$. Since 
\[\pi(A(e))\cap S_e'=\pi(S_e''),\]
$C$ is a connected component of $\pi(S_e'')$. Thus $\pi(S_d'')$ is a substack of $\pi(S_e'')$. 
It then follows from \cref{lem:subgrid-CC} that $(\pi(S_d''))_{d\in D}$ is a $(w_1,\ldots,w_n)$-grid.\\

We now show that $\bigcup_{d\in D} S_d''$ is \'etale. For ease of notation, write 
\[
S:=\bigcup_{d\in D} S_d''.
\]
Towards a contradiction, suppose that there is $x=(x_1,\ldots,x_{n+1})\in S$ such that $S$ is not \'etale at $x$.
Since $A_{(x_1,\dots,x_n)}$ is discrete, there exists an open interval $(a,b)\subseteq \RR$ such that 
\[
A_{(x_1,\dots,x_n)}\cap \big(a,b\big) = \{x_{n+1}\}.
\]
Let $e_0\in D$ be such that $x\in S_{e_0}''$.
Since $S_{e_0}''$ is a stack and $\pi(S_{e_0}'')$ is a substack of $S_{e_0}'$, there is a connected component $C$ of $S_{e_0}'$ and $f\in \mathcal{C}(C)$ such that $f(x_1,\dots,x_n) =x_{n+1}$ and $\Gamma(f)$ is a connected component of $S_{e_0}''$.\\

By \cref{lem:CC0-grid}, there is $\ep_0>0$ such that $\pi(B_{\ep_0}(x))\cap C$ is connected and 
\[
\ep_0<\min \{x_{n+1}-a,b-x_{n+1} \}.
\]
We now show that for every $\ep\in \RR_{>0}$ with $\ep<\ep_0$ there is $y=(y_1,\dots,y_{n+1})\in B_{\ep}(x)$ and a definable function $g : C \to \RR$ such that
\begin{enumerate}
    \item $g(y_1,\dots,y_{n}) = y_{n+1}$,
    \item $g(x_1,\dots,x_n) \notin \big(a,b\big)$,
    \item $\Gamma(g) \subseteq S$.
\end{enumerate}
Let $\ep\in \RR_{>0}$ with $\ep<\ep_0$. We first claim that there exists $e_1 \in D_{<e_0}$ such that
\[
 \big(B_{\ep}(x) \cap S_{e_1}''\big) \setminus \Gamma(f) \neq \emptyset
\]
Indeed, if no such $e_1$ exists, there is an open box $B$ around $x$ such that $B\cap S \subseteq \Gamma(f)$, contradicting that $S$ is not \'etale at $x$.\newline

\noindent Fix $e_1 \in D_{<e_0}$ with the above property, and let 
\[
y=(y_{1},\dots,y_{n+1})\in \big(B_{\ep}(x) \cap S_{e_1}''\big) \setminus \Gamma(f).
\]
Thus 
\[
(x_1,\dots,x_n)\in S_{e_0}' \text{ and } (y_1,\dots,y_n)\in S_{e_1}'.
\]
Because $S_{e_0}''$ is a subgrid of $S_{e_1}''$, we know that $C$ is a connected component of $S_{e_1}''$.
Since 
\[
(y_{1},\dots,y_{n})\in \pi(B_{\ep}(x) )\subseteq C,
\]
there exists $g\in \mathcal{C}(C)$ such that $y \in \Gamma (g)$ and $\Gamma(g)$ is a connected component of $S_{e_1}''$.
As $\Gamma(f)$ and $\Gamma (g)$ are both connected components of $S_{e_1}''$, we have that 
\[
\text{either } \Gamma(f)=\Gamma(g) \text{ or }\Gamma(f)\cap \Gamma(g)=\emptyset.
\]
Thus $\Gamma(f)\cap \Gamma(g)=\emptyset$, because $y\in \Gamma(g) \setminus \Gamma(f)$. 
In particular, since 
\[
g(x_1,\dots,x_n)\neq f(x_1,\dots,x_n),
\] 
we have that $g(x_1,\dots,x_n)$ is not in the interval $\big(a,b\big)$.\newline

Let $(\ep_j)_{j \in \NN}$ be a sequence of positive real numbers smaller than $\ep_0$ such that 
\begin{enumerate}
    \item [(4)]  $\pi(B_{\ep_j}(x))\cap C$ is connected for every $j\in \NN$, and
    \item [(5)] $\lim_{j\to \infty} \ep_j = 0$. 
\end{enumerate}
For each $j\in \NN$, pick $y_j=(y_{j,1},\dots,y_{j,n+1})\in B_{\ep_j}(x)$ and a definable function $g_j: C \to \RR$ satisfying conditions (1)-(3). By the pigeonhole principle, there is an infinite subset $J\subseteq \NN$ such that
\begin{itemize}
    \item $g_j(x_1,\dots,x_n) \geq b$ for all $j\in J$, or
    \item $g_j(x_1,\dots,x_n) \leq a$ for all $j\in J$.
\end{itemize}
We will handle the first case, and note that the second case follows by the same argument. So without loss of generality assume that $g_j(x_1,\dots,x_n) \geq b$ for all $j\in \NN$. \\

Set $c := x_{n+1} + \ep_0$. Since $y_j\in B_{\ep_j}(x)$ and $\ep_j\leq \ep_0$,
\[
g_j(y_{j,1},\dots,y_{j,n}) = y_{j,n+1} < c \leq b \leq g_j(x_1,\dots,x_n).
\]
Since $\pi(B_{\ep_j}(x))\cap C$ is connected and $g_j$ is continuous, by the intermediate value theorem we can choose for each $j\in \NN$ some $u_j\in \pi(B_{\ep_j}(x))\cap C$ such that $g_j(u_j)=c$. 
Since $\lim_{j\to \infty} \ep_j = 0$,
\[
\lim_{j \to \infty} (u_j,g_j(u_j)) = (x_1,\dots,x_n,c)
\]
Thus $(x_1,\dots,x_n,c)\in \cl(A)$.
Since $c\in (a,b)$ and $(a,b)\cap A_{(x_1,\dots,x_n)}=\{x_{n+1}\}$, we have that 
\[
c\in \cl(A)_{(x_1,\dots,x_n)}\setminus\cl(A_{(x_1,\dots,x_n)}).
\]
This contradicts that $\cl(A_{(x_1,\dots,x_n)})=\cl(A)_{(x_1,\dots,x_n)}$.
Therefore $S$ is \'etale and we conclude that $(S_d'')_{d\in D}$ is a $(w_1,\ldots,w_n,0)$-grid.\\

Lastly, we show that $\surf(\es'') = A \cap (\surf(\es')\times \RR)$.
Let $x\in \surf(\es'')$, and let $d\in D$ be such that $x\in S_d''$. Thus $x\in A(d)$ and $\pi(x)\in S_d'$.
Since $A(d)\subseteq A$ and $S_d'\subseteq \surf(S')$, we conclude that $x\in A \cap (\surf(\es')\times \RR)$.
Now, let $x\in A\cap (\surf(S')\times \RR)$. Then there is $d,e\in D$ such that $x\in A(d)$ and $\pi(x)\in S_e'$.
Suppose that $e\leq d$. 
Then $A(d)\subseteq A(e)$ and thus 
\[
x\in A(e)\cap (S_e' \times \RR)=S_e''.
\]
Thus $x\in \surf(\es'')$. The case $d\leq e$ can be handled similarly.

\section{Special Manifolds and Grids}\label{section:special}

Following \cite{Miller-tame}, in this section we take submanifold to mean an embedded submanifold, everywhere of the same dimension, but not necessarily connected.
A $d$-dimensional submanifold $M$ of $\RR^n$ is called \textbf{special} if there exists some coordinate projection $\mu:\RR^n\to \RR^d$ such that for every $y\in \mu(M)$ there is some open box $B$ about $y$ such that $\mu$ maps each connected component of $M\cap\mu^{-1}(B)$ homeomorphically onto $B$.
The following decomposition theorem for strongly d-minimal structures was first developed by Miller in \cite[Theorem 4]{Miller-tame}, although a flaw was found in the proof that was later remedied by Thamrongthanyalak in \cite[Theorem B]{Athipat-Michael}.

\begin{fact}\label{thm:miller-decomp}
    Suppose $\R$ is a strongly d-minimal expansion of $\overline{\RR}$.
    Let $\A$ be a finite collection of definable subsets of $\RR^n$.
    Then there exists a finite partition $\P$ of $\RR^n$ into special submanifolds compatible with $\A$ such that for each $P\in \P$ the frontier of $P$ is a finite union of elements of $\P$.
\end{fact}

There are two immediate key differences between \cref{thm:miller-decomp} and \cref{thm:intro2}.
First, we are decomposing into different objects, i.e. special submanifolds and surfaces of grids.
This difference is not just cosmetic.
For example, let $X_1$ be an open disk in $\RR^2$ of radius $R\in \RR$ and $X_2$ be a closed disk of radius $r\in \RR$ with $r<R$, both centered on the same point. 
Set $X:=X_1\setminus X_2$.
Then $X$ is an open set in $\RR^2$, thus a special submanifold.
However, $X$ is not the surface of any one grid, since to handle the hole in $X$ we would need to split the projection into multiple pieces, above which $X$ is a combination of $(0,1)$-grids and $(1,1)$-grids.
On the other hand, define the set $Y\subseteq \RR^3$ as the set of triples $(x,y,z) \in \RR^3$ such that
\[ x\in (0,1) \land y\in \{0,1\} \land (y=0 \implies z\in (0,2))\land (y=1 \implies z\in (0,1)).\]
This is a $(1,0,1)$-stack, and thus also the surface of a $(1,0,1)$-grid.
However, we claim that it is not a special submanifold.
Let $\mu:\RR^3 \to \RR^2$ be the projection onto the first and third coordinates.
Then $\mu(A)=(0,1)\times (0,2)$. Consider the point $(\frac12, 1)\in \mu(A)$.
Let $B$ be an open box around $(\frac12,1)$.
Then $\mu^{-1}(B)$ has two connected components, one of which is contained in the box $(0,1)\times \{1\} \times (0,1)$ but when projected again under $\mu$ will not cover all of $B$.
Thus $Y$ is not a $\mu$-special submanifold, and all other possible projection functions do not preserve the dimension of $Y$, and so will not make $Y$ into a special submanifold either.\\

Second, in \cref{thm:intro2} we only assume $\R$ is d-minimal, and the assumption in \cref{thm:miller-decomp} is that $\R$ is strongly d-minimal.
As previously mentioned, it is this weaker assumption on our grid decomposition theorem that allows us to deduce the equivalence of d-minimality and strong d-minimality as a result.

\bibliographystyle{abbrv}
\bibliography{refs}

\end{document}